\documentclass[a4paper,11pt]{article}
\usepackage{amssymb}
\usepackage{amsthm}
\usepackage[utf8]{inputenc}
\usepackage[T1]{fontenc}
\usepackage{lmodern}
\usepackage{graphicx}
\usepackage[a4paper]{geometry}
\usepackage{url}
\usepackage{float}
\usepackage[Rejne]{fncychap}
\usepackage{amsmath}
\usepackage[colorlinks=true,urlcolor=blue,linkcolor=blue,citecolor=red]{hyperref}
\usepackage{changepage}
\newcommand{\R}{\mathbb{R}}

\newcommand{\Conf}{\mathrm{Conf}}
\newcommand{\Conv}{\mathrm{Conv}}

\newcommand{\Mink}{\mathrm{Mink}}

\newcommand{\eeu}{\widetilde{Ein}_{1,n-1}}

\begin{document}

\newtheorem{theorem}{Theorem}

\newtheorem{definition}{Definition}

\newtheorem{lemma}{Lemma}

\newtheorem{fact}{Fact}

\newtheorem{criterion}{Criterion}

\newtheorem{remark}{Remark}

\newtheorem{example}{Example}

\newtheorem{proposition}{Proposition}

\newtheorem{corollary}{Corollary}

\begin{center}
\Large \textbf{Enveloping space of a globally hyperbolic conformally flat spacetime}
\end{center}

\begin{center}
Rym SMAÏ 
\end{center}

\paragraph*{Abstract.} We prove that any simply-connected globally hyperbolic conformally flat spacetime $V$ can be conformally embedded in a bigger conformally flat spacetime, called \textbf{enveloping space} of $V$, containing all the conformally flat Cauchy-extensions of $V$, in particular its $\mathcal{C}_0$-maximal extension. As a result, we establish a new proof of the existence and the uniqueness of the $\mathcal{C}_0$-maximal extension of a globally hyperbolic conformally flat spacetime. Furthermore, this approach allows us to prove that $\mathcal{C}_0$-maximal extensions respect inclusion.

\section{Introduction}

The notion of maximal extension of a globally hyperbolic spacetime arises from the resolution of Einstein equations in general relativity. This physical theory suggests that our universe is modelized by a Lorentzian manifold $(M,g)$ of dimension $4$ where the metric $g$ satisfies some PDEs, the so-called \emph{Einstein equations}. A way to solve them is to require that $M$ is homeomorphic to $S \times \R$ where $S$ is a Riemannian manifold. This allows to define a Cauchy problem where the initial data is the Riemannian manifold $(S,h)$ equipped with a $(2,0)$-tensor II. A solution is a Lorentzian metric $g$ on $S \times \R$ such that the restriction of $g$ to $S \times \{0\}$ is $h$ and II is the shape operator of this hypersurface. It turns out that a necessary condition to have such a solution is that $h$ and II satisfy \emph{the constraint equations}. Conversely, Choquet-Bruhat and Geroch \cite[Section 2]{Choquet-Bruhat} proved that when the constraint equations are satisfied, a local solution exists. Two natural questions arise: is it possible to extend this solution to a maximal one? If yes, is it unique up to isometry? Choquet-Bruhat and Geroch answered positively to both questions (see \cite[Theorem 3]{Choquet-Bruhat}).  

The solutions of the Cauchy problem associated to Einstein equations turn out to be \emph{globally hyperbolic} (abbrev. GH) (see Definition \ref{def: global hyperbolicity}). More generally, Geroch \cite{Geroch1970} proved that any GH spacetime admits an embedded Riemannian hypersurface which intersects every inextensible causal curve exactly once, called \emph{Cauchy hypersurface}. It turns out that all smooth Cauchy hypersurfaces of a GH spacetime are diffeomorphic. 

There is a natural partial ordering on GH spacetimes: given two GH spacetimes, $M$ and $N$, we say that $N$ is a Cauchy-extension of $M$ if there exists an isometric embedding from $M$ to $N$ sending every Cauchy hypersurface of $M$ on a Cauchy hypersurface of $N$. Such an embedding is called \emph{an isometric Cauchy-embedding}. We can ask again in this general setting the questions of the existence and the uniqueness, up to isometry, of a maximal extension. The answer to both questions is yes within a \emph{rigid category}\footnote{See e.g. \cite[Definitions 2 and 4]{Salvemini2013Maximal}. Spacetimes of constant curvature are examples of rigid categories.} of spacetimes. Actually, the spacetimes which are solution of a Cauchy problem associated to Einstein equations constitute a rigid category and it turns out that the arguments of Choquet-Bruhat and Geroch could be adapted to any other rigid category.

In this paper, we are interested in the notion of maximality in the setting of \emph{conformally flat} spacetimes of dimension $n \geq 3$. The morphisms preserving these structures are the \emph{conformal} diffeomorphisms. In \cite[Section 3.1]{Salvemini2013Maximal}, C. Rossi adapted to this setting the definition of the ordering relation on GH spacetimes by considering conformal Cauchy-embeddings instead of isometric Cauchy-embeddings. A GH conformally flat spacetime $M$ is then said to be \emph{$\mathcal{C}_0$-maximal} if any conformal Cauchy-embedding from $M$ to any GH conformally flat spacetime is surjective. C. Rossi proved that any GH conformally flat spacetime admits a $\mathcal{C}_0$-maximal extension, unique up to conformal diffeomorphism (see \cite[Sections 3.2 \& 3.3]{Salvemini2013Maximal}). Her proof is mainly based on Zorn lemma and so does not give any description of the $\mathcal{C}_0$-maximal extension. In this paper, we propose a new approach which allows us to give a constructive proof of the existence and the uniqueness of the $\mathcal{C}_0$-maximal extension. Indeed, given a simply-connected GH conformally flat spacetime $M$, we construct a bigger conformally flat spacetime $E(M)$ in which $M$ and all its conformally flat Cauchy-extensions embeds conformally. The $\mathcal{C}_0$-maximal extension of $M$ turns out to be the Cauchy development of a Cauchy hypersurface of $M$ in $E(M)$. The images in $E(M)$ of the previous embeddings satisfy the nice property of being \emph{causally convex}. A subset $U$ of a spacetime is causally convex if any causal curve joining two points of $U$ is contained in $U$. While convexity is a metric notion, causal convexity is a conformal notion. Let us add that causal convexity is a strong property in a GH spacetime: it is a classical fact that any causally convex open subset of a GH spacetime is GH.

\begin{theorem}\label{intro: enveloping space theorem}
Let $M$ be a simply-connected globally hyperbolic conformally flat spacetime. There exists a conformally flat spacetime $E(M)$ with the following properties:
\begin{enumerate}
\item $E(M)$ fibers trivially over a conformally flat Riemannian manifold $\mathcal{B}$ diffeomorphic to any Cauchy hypersurface of $M$;
\item $M$ embeds conformally in $E(M)$ as a causally convex open subset;
\item all the conformally flat Cauchy-extensions of $M$ embeds conformally in $E(M)$ as causally convex open subsets. In particular, the $\mathcal{C}_0$-maximal extension of $M$ is the Cauchy development of a Cauchy hypersurface of $M$ in $E(M)$.
\end{enumerate}
Such a spacetime $E(M)$ is called \textbf{an enveloping space} of $M$.
\end{theorem}

This result still holds for the larger class of \emph{developable} GH conformally flat spacetimes (see Definition~\ref{def: developable}). In Section \ref{sec: Causally convex subsets of E(M)}, we describe causally convex open subsets of an enveloping space $E(M)$ then, in Section \ref{sec: Eikonal functions}, we characterize those which are $\mathcal{C}_0$-maximal. \\

A consequence of Theorem \ref{intro: enveloping space theorem} is the following result:

\begin{corollary}
Any globally hyperbolic conformally flat spacetime admits a $\mathcal{C}_0$-maximal extension, unique up to conformal diffeomorphism.
\end{corollary}

Now we ask the following question. Let $V$ be a globally hyperbolic conformally flat spacetime and let $U$ be a causally convex open subset of $V$. Does the $\mathcal{C}_0$-maximal extension of $U$ embed conformally in the $\mathcal{C}_0$-maximal extension of $V$?

The $\mathcal{C}_0$-maximal extensions of $U$ and $V$ are \emph{a priori} abstract objects which depends on the Cauchy hypersurfaces of $U$ and $V$ respectively. These last ones are completely independant so the question above is not tautological. We prove in Section \ref{sec: Maximal extensions respect inclusion} that the answer is yes, in other words, that $\mathcal{C}_0$-maximal extensions preserve inclusion.

\begin{theorem}\label{intro: max. ext. respect inclusion}
Let $V$ be a globally hyperbolic conformally flat spacetime and let $U$ be a causally convex open subset of $V$. Then, the $\mathcal{C}_0$-maximal extension of $U$ is conformally equivalent to a causally convex open subset of the $\mathcal{C}_0$-maximal extension of $V$.
\end{theorem}

\subsection*{Overview of the paper} In Section \ref{sec: GH spacetimes}, we introduce the premiminary material on causality of spacetimes. We focus in particular on \emph{globally hyperbolic} spacetimes and we recall some of their main properties. Section \ref{sec: geometry of Einstein universe} deals with the model space of conformally flat Lorentzian structure, the so-called \emph{Einstein universe}. After a quick description of its geometry, we characterize \emph{causally convex open subsets} of its universal cover (see Sections \ref{sec: causally convex open subsets of Ein} and \ref{sec: duality in ein}). We devote Section \ref{sec: enveloping space} to the proof of Theorem~\ref{intro: enveloping space theorem}: we construct an enveloping space (see Section \ref{sec: def enveloping space}) and we describe its causally convex globally hyperbolic open subsets (see Section \ref{sec: Causally convex subsets of E(M)}). In Section \ref{sec: maximal ext.}, we propose a new proof of the existence and the uniqueness of the maximal extension of a globally hyperbolic conformally flat spacetime, using the notion of enveloping space. Section \ref{sec: Maximal extensions respect inclusion} is devoted to the proof of Theorem \ref{intro: max. ext. respect inclusion}. Lastly, we establish a link between the notion of $\mathcal{C}_0$-maximality and the notion of \emph{eikonal functions} in Section~\ref{sec: Eikonal functions}.


\subsection*{Acknowlegment} I would like to thank my PhD advisor Thierry Barbot for all the discussions that make this work possible, for his remarks and his help. I am also grateful to Charles Frances for his interests in my work and his relevant comments to improve a first version of this paper. This work of the Interdisciplinary Thematic Institute IRMIA++, as part of the ITI 2021-2028 program of the University of Strasbourg, CNRS and Inserm, was supported by IdEx Unistra (ANR-10-IDEX-0002), and by SFRI-STRAT’US project (ANR-20-SFRI-0012) under the framework of the French Investments for the Future Program. 

\section{Globally hyperbolic spacetimes} \label{sec: GH spacetimes}

\subsection{Preliminaries on spacetimes} 

The aim of this preliminary section is to introduce the concept of \emph{causality} in a Lorentzian manifold and briefly recall some basic causal notions as \emph{causal curves}, \emph{future and past} of points, \emph{lightcones}, \emph{achronal and acausal} subsets, etc. 

Throughout this paper, we denote by $(p,q)$ the signature of a non-degenerate quadratic form $Q$ where $p$ and $q$ are respectively the number of negative and positive coefficients in the polar decomposition of $Q$.

\paragraph{Spacetimes.} A \emph{Lorentzian metric} on a manifold of dimension $n$ is a non-degenerate symmetric $2$-tensor $g$ of signature $(1,n-1)$. A manifold equipped with a Lorentzian is called \emph{Lorentzian}. 

In a Lorentzian manifold $(M,g)$, we say that a non-zero tangent vector $v$ is \emph{timelike, lightlike, spacelike} if $g(v,v)$ is respectively negative, zero, positive. The set of timelike vectors is the union of two convex open cones. When it is possible to make a continuous choice of a connected component in each tangent space, the manifold $M$ is said \emph{time-orientable}. The timelike vectors in the chosen component are said \emph{future-directed} while those in the other component are said \emph{past-directed}. A \emph{spacetime} is an oriented and time-oriented Lorentzian manifold. 

\paragraph{Future, past.} In a spacetime $M$, a differential curve is \emph{timelike, lightlike, spacelike} if its tangent vectors are timelike, lightlike, spacelike. It is \emph{causal} if its tangent vectors are either timelike or lightlike. 

Given a point $p$ in $M$, the \emph{future} (resp. \emph{chronological future}) of $p$, denoted $J^+(p)$ (resp. $I^+(p)$), is the set of endpoints of future-directed causal (resp. timelike) curves starting from $p$. More generally, the future (resp. chronological future) of a subset $A$ of $M$, denoted $J^+(A)$ (resp. $I^+(A)$), is the union of $J^+(a)$  (resp. $I^+(a)$) where $a \in A$. 

An open subset $U$ of $M$ is a spacetime and the intrinsic causality relations of $U$ imply the corresponding ones in $M$. We denote $J^+(A,U)$ (resp. $I^+(A,U)$) the future (resp. chronological future) in the manifold $U$ of a set $A \subset U$. Then, $I^+(A, U) \subset I^+(A) \cap U$.

Dual to the preceding definitions are corresponding \emph{past} versions. In general, \emph{past} definitions and proofs follows from  future versions (and vice versa) by reversing time-orientation.

\paragraph{Diamonds.} We call \emph{diamond of $M$} any intersection $J^-(p) \cap J^+(q)$, where $p, q \in M$ such that $p \in J^+(q)$. We denote it $J(p,q)$. Given two points $p, q \in M$ such that $p \in I^+(q)$, the interior of the diamond $J(p,q)$ is the intersection $I^-(p) \cap I^+(q)$ and is denoted $I(p,q)$ (see \cite[Lemma 6, p.404]{oneill}).

\paragraph{Achronal, acausal subsets.} A subset $A$ of a spacetime $M$ is called \emph{achronal} (resp. \emph{acausal}) if no timelike (resp. causal) curve intersects $A$ more than once.

\paragraph{Causal convexity.} In Riemannian geometry, it is often useful to consider open neighborhoods which are \emph{geodesically convex}. In Lorentzian geometry, there is, in addition, \emph{a causal convexity} notion.  A subset $U$ of $M$ is said \emph{causally convex} if for every $p, q \in U$, any causal curve of $M$ joining $p$ to $q$ is contained in $U$. Equivalently, if every diamond $J(p,q)$ of $M$ with $p, q \in U$ is contained in $U$. It is easy to check that the intersection of two causally convex subsets is causally convex.

\paragraph{Cauchy developments.} Let $A$ be an achronal subset of $M$. \emph{The future} (resp. \emph{past}) \emph{Cauchy development} of $A$, denoted $\mathcal{C}^+(A)$ (resp. $\mathcal{C}^-(A)$), is the set of points $p$ of $M$ such that every past-inextensible (resp. future-inextensible) causal curve through $p$ meets $A$. \emph{The Cauchy development of $A$} is the union of $\mathcal{C}^+(A)$ and $\mathcal{C}^-(A)$, denoted $\mathcal{C}(A)$.

\subsection{Global hyperbolicity}

\begin{definition}
A spacetime $M$ is said \emph{strongly causal} if for every point $p \in M$ and every neighborhood $U$ of $p$, there exists a neighborhood $V$ of $p$ contained in $U$, which is causally convex in $M$.
\end{definition}

\begin{definition}\label{def: global hyperbolicity}
A spacetime $M$ is said \emph{globally hyperbolic} (abbrev. GH) if the two following conditions hold:
\begin{enumerate}
\item $M$ is strongly causal.
\item all diamonds of $M$ are compact.
\end{enumerate}
\end{definition}

It was proved by Sanchez in \cite{bernal2007globally} that the first condition can be weaken to \emph{$M$ is causal}, that is $M$ contains no causal loop.

A classical result of Geroch \cite{Geroch1970}, later improved by Bernal and Sanchez \cite{BernalSanchez}, gives a characterization of global hyperbolicity involving the notion of \emph{Cauchy hypersurface}.

\begin{definition}
A \emph{topological} (resp. \emph{smooth}) \emph{Cauchy hypersurface} is an achronal topological hypersurface (resp. an embedded Riemannian hypersurface) that is met by exactly once by every inextensible causal curve of $M$.
\end{definition}

\begin{theorem}[\cite{Geroch1970}]
A spacetime $M$ is globally hyperbolic if and only if it contains a topological Cauchy hypersurface.
\end{theorem}

Bernal and Sanchez \cite{BernalSanchez} improved this result by proving the existence of a \emph{smooth} Cauchy hypersurface.

Topological (resp. smooth) Cauchy hypersurfaces of a globally hyperbolic spacetime are homeomorphic (resp. diffeomorphic). Therefore, one can set the following definition.

\begin{definition}
A globally hyperbolic spacetime is said \emph{Cauchy-compact} (or \emph{spatially compact}) if it admits a compact Cauchy hypersurface.
\end{definition}

A remarkable property of globally hyperbolic spacetimes is that causal convexity implies global hyperbolicity.

\begin{proposition}\label{prop: causally convex implies GH}
Let $M$ be a globally hyperbolic spacetime. Then, any causally convex open subset of $M$ is globally hyperbolic.
\end{proposition}

\begin{proof}
Since $M$ is globally hyperbolic, there is no causal loop in $U$. Since $U$ is causally convex, the diamonds of $U$ are exactly the diamonds of $M$ contained in $U$. Thus, they are compact.
\end{proof}

\subsection{Shadows}

In this section, we show that the causal structure of globally hyperbolic spacetimes is encoded by compact subsets of a Cauchy hypersurface, called \emph{shadows.}\footnote{This terminology has been introduced by C. Rossi in her thesis \cite[Chapitre 4]{salveminithesis}.} 

Let $M$ be a globally hyperbolic spacetime and let $S \subset M$ be a Riemannian Cauchy hypersurface.

\begin{definition}\label{def: shadow}
Let $p \in M$. We call \emph{shadow of $p$ on $S$}, denoted by $O(p,S)$, the set of points in $S$ which are causally related to $p$. When there is no confusion about the Cauchy hypersurface $S$, we will simply write $O(p)$ instead of $O(p,S)$. 
\end{definition}

If $p \in I^\pm(S)$ , then $O(p,S) = J^\mp(p) \cap S$; if $p \in S$, $O(p,S)$ is reduced to $\{p\}$. Thus, by \cite[Lemma 40, p.423]{oneill}), shadows are \emph{compact}. 

The main interest of the notion \emph{shadows} is given by the following proposition proved by C. Rossi in her thesis (see \cite[Prop. 2.6, chap. 4]{salveminithesis}).

\begin{proposition}\label{prop: shadows}
Suppose $S$ is not compact. Then, two points $p$ and $q$ of $M$ in the chronological future of $S$ coincide if and only if their shadows on $S$ are equal. \qed
\end{proposition}

By Proposition \ref{prop: shadows}, the shadows on $S$ characterize completely the points of the globally hyperbolic spacetime $M$. This allows to reduce, in some situations, the study of the spacetime to the study of compact subsets of a Riemannian manifold.

\section{Geometry of Einstein universe} \label{sec: geometry of Einstein universe}

In this section, we introduce the model space of conformally flat Lorentzian structures, the so-called \emph{Einstein universe}, and we describe its causal structure.

\subsection{The Klein model}

Let $\R^{2,n}$ be the vector space $\R^{n+2}$ of dimension $(n+2)$ equipped with the nondegenerate quadratic form $q_{2,n}$ of signature $(2,n)$ given by
\begin{align*}
q_{2,n}(u,v,x_1,\ldots, x_n) &= -u^2 - v^2 + x_1^2 + \ldots + x_n^2
\end{align*}
in the coordinate system $(u,v,x_1,\ldots,x_n)$ associated to the canonical basis of $\R^{n+2}$.

\begin{definition}
\emph{Einstein universe} of dimension $n$, denoted by $\mathsf{Ein}_{1,n-1}$, is the space of isotropic lines of $\R^{2,n}$ with respect to the quadratic form $q_{2,n}$, namely
\begin{align*}
\mathsf{Ein}_{1,n-1} &= \{[x] \in \mathbb{P}(\R^{2,n}):\ q_{2,n}(x) = 0\}.
\end{align*}
\end{definition}

In practice, it is more convenient to work with the double cover of the Einstein universe, denoted by $Ein_{1,n-1}$:
\begin{align*}
Ein_{1,n-1} &= \{[x] \in \mathbb{S}(\R^{2,n}):\ q_{2,n}(x) = 0\}
\end{align*}
where $\mathbb{S}(\R^{2,n})$ is the sphere of rays, namely the quotient of $\R^{2,n} \backslash \{0\}$ by positive homotheties.

\subsection{Spatio-temporal decomposition of Einstein universe}

The choice of a timelike plane of $\R^{2,n}$, i.e. a plane on which the restriction of $q_{2,n}$ is negative definite, defines a spatio-temporal decomposition of Einstein universe:

\begin{lemma}
Any timelike plane $P \subset \R^{2,n}$ defines a diffeomorphism between $\mathbb{S}^{n-1} \times \mathbb{S}^1$ and $Ein_{1,n-1}$.
\end{lemma}

\begin{proof}
Consider the orthogonal splitting $\R^{2,n} = P^\perp \oplus P$ and call $q_{P^\perp}$ and $q_{P}$ the positive definite quadratic form induced by $\pm q_{2,n}$ on $P^\perp$ and $P$ respectively. The restriction of the canonical projection $\R^{2,n} \backslash \{0\}$ on $\mathbb{S}(\R^{2,n})$ to the set of points $(x,y) \in P^\perp \oplus P$ such that $q_{P^\perp}(x) = q_P(y) = 1$ defines a map from $\mathbb{S}^{n-1} \times \mathbb{S}^1$ to $Ein_{1,n-1}$. It is easy to check that this map is a diffeomorphism.
\end{proof}

For every timelike plane $P \subset \R^{2,n}$, the quadratic form $q_{2,n}$ induces a Lorentzian metric $g_P$ on $\mathbb{S}^{n-1} \times \mathbb{S}^1$ given by
\begin{align*}
g_P &= d\sigma^2 (P) - d\theta^2(P)
\end{align*}
where $d\sigma^2(P)$ is the round metric on $\mathbb{S}^{n-1} \subset (P^\perp,q_{P^\perp})$ induced by $q_{P^\perp}$  and $d\theta^2(P)$ is the round metric on $\mathbb{S}^{1} \subset (P,q_P)$ induced by $q_P$.

An easy computation shows that if $P' \subset \R^{2,n}$ is another timelike plane, the Lorentzian metric $g_{P'}$ is conformally equivalent to $g_P$, i.e. $q_P$ and $g_P'$ are proportionnal by a positive smooth function on $\mathbb{S}^{n-1} \times \mathbb{S}^1$. As a result, Einstein universe is naturally equipped with a conformal class of Lorentzian metrics. This Lorentzian conformal structure induces causality on Einstein universe. Indeed, changing the metric in the conformal class consists in multiplying by a positive function and so does not change the sign of the norm of a tangent vector. The causal structure of Einstein universe is trivial: any point is causally related to any other one (see e.g. \cite[Cor. 2.10, Chap. 2]{salveminithesis}).

Let us point out that in general geodesics are not well-defined in a conformal spacetime. Indeed, a computation of the Levi Civita connexion shows that geodesics are not preserved by conformal changes of metrics. Nevertheless, lightlike geodesics are preserved as \emph{non-parametrized} curves (see e.g. \cite[Théorème 3]{francesarticle}). 

\subsection{Causal structure of the universal cover} \label{sec: euu}

Let $\eeu$ be the universal cover of $Ein_{1,n-1}$. When $n \geq 3$, every diffeomorphism between $Ein_{1,n-1}$ and $\mathbb{S}^{n-1} \times \mathbb{S}^1$ lifts to a diffeomorphism between $\eeu$ and $\mathbb{S}^{n-1} \times \R$. The pull-back by the projection $\mathbb{S}^{n-1} \times \R \to \mathbb{S}^{n-1} \times \mathbb{S}^1$ of the conformal class $[d\sigma^2 - d\theta^2]$ on $\mathbb{S}^{n-1} \times \mathbb{S}^1$ defined previously is the conformal class of the Lorentzian metric $d\sigma^2 - dt^2$ where $dt^2$ is the usual metric on $\R$. This induces a natural conformally flat Lorentzian structure on $\eeu$. 

\begin{definition}
We call \emph{spatio-temporal decomposition of $\eeu$} any conformal diffeomorphism between $\eeu$ and $\mathbb{S}^{n-1} \times \R$. 
\end{definition}

In what follows, we fix a spatio-temporal decomposition and we identify $\eeu$ to $\mathbb{S}^{n-1} \times \R$.\\

The fundamental group of $Ein_{1,n-1}$ is isomorphic to $\mathbb{Z}$, generated by the transformation $\delta: \eeu \to \eeu$ defined by $\delta(x,t) = (x, t + 2\pi)$. That of $\mathsf{Ein}_{1,n-1}$ is generated by the transformation $\sigma: \eeu \to \eeu$ such that $\sigma^2 = \delta$, i.e. the map defined by $\sigma(x,t) = (-x, t + \pi)$.

\begin{definition}
Two points $p$ and $q$ of $\eeu$ are said to be \emph{conjugate} if one is the image under $\sigma$ of the other.
\end{definition}

While the causal structure of $Ein_{1,n-1}$ is trivial, the causal structure of $\eeu$ is rich. We give a brief description below. We direct to \cite[Chap. 2]{salveminithesis} for more details.\\

Lightlike geodesics of $\eeu$ are the curves which can be written, up to reparametrization, as $(x(t), t)$ where $x: I \to \mathbb{S}^{n-1}$ is a geodesic of $\mathbb{S}^{n-1}$ defined on an interval $I$ of $\R$. The inextensible ones are those for which $x$ is defined on $\R$. 

It turns out that the photons going through a point $(x_0, t_0)$ have common intersections at the points $\sigma^k(x_0,t_0)$, for $k \in \mathbb{Z}$; and are pairwise disjoint outside these points. The lightcone of a point $(x_0, t_0)$ is the set of points $(x,t)$ such that $d(x,x_0) = |t - t_0|$ where $d$ is the distance on the sphere $\mathbb{S}^{n-1}$ induced by the round metric. It disconnects $\eeu$ in three connected components:
\begin{itemize}
\item The \emph{chronological future} of $(x_0, t_0)$: this is the set of points $(x,t)$ of $\mathbb{S}^{n-1} \times \R$ such that $d(x,x_0) < t - t_0$.
\item The \emph{chronological past} of $(x_0, t_0)$: this is the set of points $(x,t)$ of $\mathbb{S}^{n-1} \times \R$ such that $d(x,x_0) < t_0 - t$.
\item The set of points non-causally related to $(x_0, t_0)$, i.e. the set of points $(x,t)$ of $\mathbb{S}^{n-1} \times \R$ such that $d(x,x_0) > |t - t_0|$. This is exactly the interior of the diamond of vertices $\sigma(x_0, t_0)$ and $\sigma^{-1}(x_0,t_0)$. It is conformally diffeomorphic to Minkowski spacetime (see e.g. \cite[Lemma 2.38 and Corollary 2.43]{smai2022anosov}) and is called \emph{affine chart}. We denote it $\Mink_0(x_0,t_0)$.
\end{itemize}

There are two other affine charts associated to the point $(x_0,t_0)$, namely:
\begin{itemize}
\item the set of points non-causally related to $\sigma(x_0, t_0)$, contained in the chronological future of $(x_0, t_0)$, denoted $\Mink_+(x_0,t_0)$;
\item the set of points non-causally related to $\sigma^{-1}(x_0,t_0)$, contained in the chronological past of $(x_0, t_0)$, denoted $\Mink_-(x_0,t_0)$. 
\end{itemize}

The universal cover $\eeu$ is globally hyperbolic: any sphere $\mathbb{S}^{n-1} \times \{t\}$, where $t \in \R$, is a Cauchy hypersurface.

\subsection{Conformal group} 

The subgroup $O(2,n) \subset Gl_{n+2}(\R)$ preserving $q_{2,n}$, acts conformally on $Ein_{1,n-1}$. When $n \geq 3$, the conformal group of $Ein_{1,n-1}$ is \emph{exactly} $O(2,n)$. This is a consequence of the following result, which is an extension to Einstein universe, of a classical theorem of Liouville in Euclidean conformal geometry (see e.g. \cite{francesarticle}):

\begin{theorem}\label{Liouville theorem}
Let $n \geq 3$. Any conformal transformation between two open subsets of $Ein_{1,n-1}$ is the restriction of an element of $O(2,n)$.
\end{theorem}

It is a classical fact that every conformal diffeomorphism of $Ein_{1,n-1}$ lifts to a conformal diffeomorphism of $\eeu$. Conversely, by Theorem \ref{Liouville theorem}, every conformal transformation of $\eeu$ defines a unique conformal transformation of the quotient space $Ein_{1,n-1} = \eeu / <\delta>$.

Let $\Conf(\eeu)$ denote the group of conformal transformations of $\eeu$. Let $j: \Conf(\eeu) \longrightarrow O(2,n)$ be the natural projection. This is a surjective group morphism whose kernel is generated by $\delta$. 

\subsection{Causally convex open subsets of Einstein universe} \label{sec: causally convex open subsets of Ein}

In this section, we characterize causally convex open subsets of $\eeu$ in a spatio-temporal decomposition $\mathbb{S}^{n-1} \times \R$. We denote by $d$ the distance on $\mathbb{S}^{n-1}$ induced by the round metric.

\begin{proposition} \label{lemma: causally convex in ein 1}
Let $\Omega$ be a causally convex open subset of $\eeu$. Then, there exist two $1$-Lipschitz functions $f^+$ and $f^-$ from an open subset $U$ of $\mathbb{S}^{n-1}$ to $\overline{\R}$ such that the following hold:
\begin{itemize}
\item $f^- < f^+$ on $U$;
\item the extensions of $f^+$ and $f^-$ to $\partial U$ coincide.
\item $\Omega$ is the set of points $(x,t)$ of $\eeu$ such that $f^-(x) < t < f^+(x)$.
\end{itemize}
\end{proposition}

The proof of Proposition \ref{lemma: causally convex in ein 1} uses the following lemma.

\begin{lemma} \label{lemma: the closure of a causally convex is causally convex}
For every points $p$ and $q$ in the closure of $\Omega$ such that $p \in I^+(q)$, the intersection $I^-(p) \cap I^+(q)$ is contained in $\Omega$. 
\end{lemma}

\begin{proof}
Let $p, q \in \overline{\Omega}$ such that $p \in I^+(q)$. There exist two sequences $\{p_i\}$ and $\{q_i\}$ of elements of $\Omega$ such that $\lim p_i = p$ and $\lim q_i = q$. Let $r \in I(p,q)$. Then, $I^+(r)$ is an open neighborhood of $p$ and $I^-(r)$ is an open neighborhood of $q$. As a result, there exists an integer $i_0$ such that $p_{i_0} \in I^+(r)$ and $q_{i_0} \in I^-(r)$. It follows that $r \in I(p_{i_0}, q_{i_0}) \subset \Omega$.
\end{proof}

\begin{proof}[Proof of Proposition \ref{lemma: causally convex in ein 1}]
Let $U$ be the projection of $\Omega$ on the sphere $\mathbb{S}^{n-1}$. Since $\Omega$ is causally convex, the intersection of $\Omega$ with any timelike line $\{x\} \times \R$, where $x \in U$, is connected, i.e. it is a segment $\{x\} \times ]f^-(x), f^+(x)[$. This defines two functions $f^+$ and $f^-$ from $U$ to $\overline{\R}$ such that $\Omega$ is the set of points $(x,t)$ such that $f^-(x) < t < f^+(x)$.

\begin{description}
\item[Fact $1$.] \emph{If there exists $x \in U$ such that $f^+(x) = +\infty$ then $f^+ \equiv + \infty$.} 

This is equivalent to prove that $\Omega$ is future-complete, i.e. $I^+(\Omega) \subset \Omega$, as soon as it contains a timelike half-line $\alpha= \{x\} \times [t, +\infty[$. Let $p \in \Omega$ and let $q \in I^+(p)$. Since $\alpha$ is future-inextensible, it intersects $I^+(q)$. Let $q'$ be a point in this intersection. Then, $q \in J(q',p')$ where $p' \in \Omega \cap I^-(p)$. Thus, $q \in \Omega$.

\item[Fact $2$.] \emph{If $f^+$ is finite, it is $1$-Lipschitz.}

This is equivalent to prove that the graph of $f^+$ is achronal. Suppose there exist two distinct points $p, q$ in the graph of $f^+$ such that $p \in I^+(q)$. Since $p \in \partial \Omega$, we have $I^+(q) \cap \Omega \neq \emptyset$. Then, $q \in I(p',q')$ where $p' \in I^+(q) \cap \Omega$ and $q' \in \Omega \cap I^-(q)$. Hence, $q \in \Omega$. Contradiction.

\item[Fact $3$.] \emph{If $f^+$ and $f^-$ are finite and $\partial U$ is non-empty, then the extensions of $f^+$ and $f^-$ to $\partial U$ are equal.}

Let $\overline{f^+}$ (resp. $\overline{f^-}$) be the extension of $f^+$ (resp. $f^-$) to $\overline{U}$. Suppose there exists $x \in U$ such that $\overline{f^-}(x) < \overline{f^+}(x)$. Then, the timelike segment $\{x\} \times ]\overline{f^-}(x), \overline{f^+}(x)[$ is contained in the boundary of $\Omega$. This contradicts Lemma \ref{lemma: the closure of a causally convex is causally convex}.
\end{description}
\end{proof}

Conversely:

\begin{proposition} \label{lemma: causally convex in ein 2}
Let $f^+, f^- : U \subset \mathbb{S}^{n-1} \to \overline{\R}$ be two $1$-Lipschitz functions defined on an open subset $U$ of $\mathbb{S}^{n-1}$ such that:
\begin{itemize}
\item $f^- < f^+$ on $U$;
\item the extensions of $f^+$ and $f^-$ to $\partial U$ coincide.
\end{itemize}
Then, the set of points $(x,t)$ of $\eeu$ such that $f^-(x) < t < f^+(x)$, named $\Omega$, is causally convex in $\eeu$.
\end{proposition}

The proof of Proposition \ref{lemma: causally convex in ein 2} uses the following lemma.

\begin{lemma} \label{lemma: characterisation of the boundary of a causally convex subset of Ein}
Let $\overline{f^+}$ be the extension of $f^+$ to $\overline{U}$. Then, for every point $p$ in the graph of $\overline{f^+}$, the causal future of $p$ is disjoint from $\Omega$.
\end{lemma}

\begin{proof}
Let $x \in U$ and set $p = (x, f^+(x))$. Suppose there exists $(y,s) \in \Omega \cap J^+(p)$. Then, $d(x,y) \leq s - f^+(x) < f^+(y) - f^+(x) \leq d(x,y)$. Contradiction. Then, $J^+(p) \cap \Omega = \emptyset$.

Now, let $x \in \partial U$ and let $\{x_i\}$ be a sequence of elements of $U$ such that $x = \lim x_i$. Set $p_i = (x_i, f^+(x_i))$. Suppose there exists $q \in J^+(p) \cap \Omega$. Since $\Omega$ is open, there exists $q' \in I^+(q) \cap \Omega$. By transitivity, $q' \in I^+(p)$. Then, $I^-(q')$ is an open neighborhood of $p$. Since $\lim p_i = p$, we deduce that $p_i \in I^-(q')$ for $i$ big enough. Equivalently, $q' \in I^+(p_i)$. Thus, $I^+(p_i) \cap \Omega \neq \emptyset$. Contradiction. 
\end{proof}

There is an analogue statement for the extension of $f^-$ to $\overline{U}$, denoted $\overline{f^-}$, with the reverse time-orientation.

\begin{proof}[Proof of Proposition \ref{lemma: causally convex in ein 2}]
Since $f^\pm$ is $1$-Lipschitz, if $f^\pm$ is infinite in a point of $U$, it is infinite on $U$. If $f^+ \equiv + \infty$ and $f^- \equiv - \infty$, we have $U = \mathbb{S}^{n-1}$ and $\Omega = \eeu$. 

Suppose $f^+ < +\infty$ and $f^- \equiv -\infty$. Let $p, q \in \Omega$ such that $q \in J^+(p)$. Let $\gamma$~be a future causal curve of $\eeu$ joining $p$ to $q$. Suppose that $\gamma \nsubseteq \Omega$. Then, $\gamma$ intersects the boundary of $\Omega$, reduced in this case to the graph of $f^+$, in a point $r = (x, f^+(x))$ where $x \in U$. By Lemma \ref{lemma: characterisation of the boundary of a causally convex subset of Ein}, $J^+(r)$ is disjoint from $\Omega$. Then, the segment of $\gamma$ joining $r$ to $q$ is contained in $J^+(r)$. Thus, $q \not \in \Omega$. Contradiction.

Suppose now that $f^+$ and $f^-$ are finite. If $\partial U$ is empty, the proof is similar to the previous case. Otherwise, we call $f$ the common extension of $f^+$ and $f^-$ to $\partial U$. In this case, the boundary of $\Omega$ is the union of the graphs of $f^+, f^-$ and $f$. By Lemma \ref{lemma: characterisation of the boundary of a causally convex subset of Ein}, the points of $\Omega$ are not causally related to any point in the graph of $f$. Therefore, the previous arguments still hold.
\end{proof}

Now, we describe Cauchy hypersurfaces of causally convex open subsets of $\eeu$. Let 
\begin{align*}
\Omega &:= \{(x,t) \in U \times \R:\ f^-(x) < t < f^+(x)\}
\end{align*}
be a causally convex open subset of $\eeu$ where $U$ is an open subset of $\mathbb{S}^{n-1}$ and $f^+, f^-$ are the functions from $U$ to $\overline{\R}$ given by Proposition \ref{lemma: causally convex in ein 1}.

\begin{proposition}\label{prop: Cauchy hypersurfaces of causally convex subsets of Ein}
Let $h$ be a $1$-Lipschitz real-valued function defined on $U$ such that its extension to $\partial U$ coincide with that of $f^+$ and $f^-$ and $f^- < h < f^+$ on $U$. Then, the graph of $h$ is a topological Cauchy hypersurface of $\Omega$.
\end{proposition}

Proposition \ref{prop: Cauchy hypersurfaces of causally convex subsets of Ein} is a consequence of the following lemma.

\begin{lemma}\label{lemma: causal boundary of causally convex subsets of Ein}
Suppose $f^+$ and $f^-$ are finite. Then, every inextensible timelike curve of $\eeu$ that intersects $\Omega$ meet each of the graphs of $f^+$ and $f^-$.
\end{lemma}

\begin{proof}
Let $\gamma$ be an inextensible timelike curve of $\eeu$ that intersects $\Omega$. Then, $\gamma$ intersects the boundary of $\Omega$. If $\partial U$ is empty, $\partial \Omega$ is the union of the graphs of $f^+$ and $f^-$. Otherwise, $\partial \Omega$ is the union of the graphs of $f^+$, $f^+$ and $f$ where $f$ is the common extension of $f^+$ and $f^-$ to $\partial U$. Since the points of $\Omega$ are not causally related to any point of the graph of $f$ (see Lemma \ref{lemma: characterisation of the boundary of a causally convex subset of Ein}), we deduce that in both cases, $\gamma$ intersects the graph of $f^+$ or the graph of $f^-$.

Suppose $\gamma$ meets the graph of $f^+$. Since $\Omega$ is not past-complete, $\gamma$ leaves $\Omega$ and so intersects again the its boundary. Since the graph of $f^+$ is achronal, $\gamma$ could not intersects the graph of $f^+$ a second time. Thus, $\gamma$ intersects the graph of $f^-$.
\end{proof}

\begin{proof}[Proof of Proposition \ref{prop: Cauchy hypersurfaces of causally convex subsets of Ein}]
Let $\gamma$ be an inextensible timelike curve of $\Omega$. Since $\Omega$ is causally convex, $\gamma$ is the intersection of $\Omega$ with an inextensible timelike curve $\tilde{\gamma}$ of $\eeu$. By Lemma \ref{lemma: causal boundary of causally convex subsets of Ein}, $\tilde{\gamma}$ meets the graph of $h$. Moreover, since the graph of $h$ is achronal, $\tilde{\gamma}$ intersects it exactly once. The proposition follows from \cite[Definition 28 and Lemma 29, p. 415]{oneill}.
\end{proof}

\subsection{Duality} \label{sec: duality in ein}

In this section, we highlight a particular class of causally convex open subsets of $\eeu$ involving a notion of \emph{duality} in Einstein universe.

\paragraph{Duality in the Klein model.} Recall that a subset of $\mathbb{S}(\R^{2,n})$ is said to be \emph{convex} if it is the projectivization of a convex subset of $\R^{2,n}$. The convex hull of a subset $A$ of $\mathbb{S}(\R^{2,n})$ is the smallest convex containing $A$. 

Let $\Lambda \subset Ein_{1,n-1}$. Let us denote $\Conv(\Lambda)$ the convex hull of $\Lambda$ in $\mathbb{S}(\R^{2,n})$. The \emph{dual convex cone} of $\Lambda$ in $\mathbb{S}(\R^{n+2})$ is
\begin{align*}
\Conv^*(\Lambda) &= \{\mathrm{x} \in \mathbb{S}(\R^{2,n}):\ <\mathrm{x}, \mathrm{y}>_{2,n} < 0\ \forall \mathrm{y} \in \Conv(\Lambda)\}.
\end{align*}

\begin{definition}
We call \emph{dual of $\Lambda$} the intersection of $Ein_{1,n-1}$ with the dual cone $\Conv^*(\Lambda)$.
\end{definition}

Notice that 
\begin{align*}
\Conv^*(\Lambda) \cap Ein_{1,n-1} &= \{\mathrm{x} \in Ein_{1,n-1}:\ <\mathrm{x}, \mathrm{y}>_{2,n} < 0\ \forall \mathrm{y} \in \Lambda\}.
\end{align*} 

\paragraph{Duality in the universal cover.} Let $\pi: \eeu \to Ein_{1,n-1}$ be the universal covering map. 

\begin{lemma}\label{lemma: duality}
Let $\Lambda \subset \eeu$. The restriction of the projection $\pi$ to the set of points which are non-causally related to any point of $\Lambda$ is injective. Furthermore, its image is contained in the dual of the projection of $\Lambda$ in $Ein_{1,n-1}$. If in addition, $\Lambda$ is acausal, we have equality.
\end{lemma}

\begin{proof}
Set $\Omega := \eeu \backslash (J^+(\Lambda) \cup J^-(\Lambda))$. By definition, $\Omega$ is the intersection of the affine charts $\Mink_0(p)$ where $p \in \Lambda$. As a result, the restriction of $\pi$ to $\Omega$ is injective and its image is contained in the dual of $\pi(\Lambda)$ (see \cite[Corollary 2.43]{smai2022anosov}). 

Suppose $\Lambda$ is acausal. Then, $\pi(\Lambda)$ is \emph{negative}, i.e. for every $\mathrm{x}, \mathrm{y} \in \pi(\Lambda)$, we have $<x, y>_{2,n} < 0$ for every representant $x, y \in \R^{2,n}$ of $\mathrm{x}, \mathrm{y}$ (see \cite[Lemma 10.13]{andersson2012}). 

Let $\mathrm{x} \in Ein_{1,n-1}  \cap \Conv^*(\pi(\Lambda))$. Set $\widehat{\Lambda_0} :=  \{x\} \cup \pi(\Lambda)$. By definition, $\widehat{\Lambda_0}$ is a negative subset of $Ein_{1,n-1}$. By \cite[Proposition 2.47]{smai2022anosov}, there exists an acausal subset $\Lambda_0$ of $\eeu$ which projects on $\widehat{\Lambda_0}$. Furthermore, the proof of \cite[Proposition 2.47]{smai2022anosov} shows that we can choose such a $\Lambda_0$ such that it contains $\Lambda$. As a consequence, $\Lambda_0$ is the union of a lift $p$ of $\mathrm{x}$ and $\Lambda$. Since $\Lambda_0$ is acausal, $p$ is non-causally related to any point of $\Lambda$. The lemma follows. 
\end{proof}

Lemma \ref{lemma: duality} motivates the following definition.

\begin{definition}
Let $\Lambda$ be a subset of $\eeu$. We call \emph{dual} of $\Lambda$, denoted by $\Lambda^\circ$, the set of points which are non-causally related to any point of $\Lambda$. 
\end{definition}

\begin{lemma}\label{lemma: the dual of a closed subset is open}
Let $\Lambda$ be a subset of $\eeu$ such that its dual is non-empty. Then, the dual of $\Lambda$ is causally convex. If in addition $\Lambda$ is closed, its dual is open.
\end{lemma}

\begin{proof}
Let $p$ and $q$ be two points in the dual $\Lambda^\circ$, joined by a causal curve~$\gamma: I \subset \R \to M$. Suppose there is $t \in I$ such that $\gamma(t) \not \in \Lambda^\circ$, in other words $\gamma(t)$ is causally related to a point $\lambda \in \Lambda$. By transitivity, it follows that $p$ or $q$ is causally related to $\lambda$. Contradiction.

Suppose $\Lambda$ is closed. If $\Lambda$ is not compact, it would contain a causal curve inextensible in the future or in the past. Then, $J^+(\Lambda) \cup J^-(\Lambda)$ would be the whole space $Ein_{1,n-1}$ and $\Lambda^\circ$ would be empty. Therefore, $\Lambda$ is compact. It follows that $J^\pm(\Lambda)$ is closed. Hence, $\Lambda^\circ$ is open. 
\end{proof}

Now, we characterize duals of achronal closed subsets of $\eeu$ in a spatio-temporal decomposition $\mathbb{S}^{n-1} \times \R$.\\

Let $\Lambda$ be a closed achronal subset of $\eeu$. It is the graph of a $1$-Lipschitz real-valued function $f$ defined on a closed subset $\Lambda_0$ of the sphere $\mathbb{S}^{n-1}$. Let $f^+, f^-$ be the real-valued functions defined for every $x \in \mathbb{S}^{n-1}$ by:
\begin{align*}
f^+(x) &= \inf_{x_0 \in \Lambda_0} \{f(x_0) + d(x,x_0)\} \\
f^-(x) &= \sup_{x_0 \in \Lambda_0} \{f(x_0) - d(x,x_0)\}.
\end{align*}
Notice that $f^-(x) < f^+(x)$ for every $x \in \mathbb{S}^{n-1} \backslash \Lambda_0$ and that $f^+$ and $f^-$ are equal to $f$ on~$\Lambda_0$.

\begin{proposition} \label{prop: characterization of the dual}
The dual of $\Lambda$ is the set of points $(x,t)$ of $\eeu$ such that $f^-(x) < t < f^+(x)$.
\end{proposition}

\begin{proof}
Let $(x,t) \in \mathbb{S}^{n-1} \times \R$ be a point in the dual of $\Lambda$. By definition, $(x,t)$ is non-causally related to any point $(x_0, f(x_0))$ where $x_0 \in \Lambda_0$. In other words, for every $x_0 \in \Lambda_0$, we have $d(x,x_0) > |t - f(x_0)|$, i.e. $f(x_0) - t < d(x,x_0) < f(x_0) + t$. Hence,
\begin{align*}
\sup_{x_0 \in \Lambda_0} \{f(x_0) - d(x,x_0)\} \leq t \leq \inf_{x_0 \in \Lambda_0} \{f(x_0) + d(x,x_0)\}.
\end{align*}
Since $\Lambda_0$ is compact, the supremum and infimum above are attained; the previous inequalities are then strict. Thus, we obtain $f^-(x) < t < f^+(x)$. The converse inclusion is clear.
\end{proof}

\section{Enveloping space of a simply connected GH conformally flat spacetime} \label{sec: enveloping space}

\subsection{Conformally flat spacetimes}

A spacetime is called \emph{conformally flat} if it is locally conformal to Minkowski spacetime. Einstein universe is conformally flat since any point of Einstein universe admits a neighborhood conformally equivalent to Minkowski spacetime (see Section \ref{sec: euu}). It follows that any spacetime locally modeled on Einstein universe, i.e. equipped with a $(O(2,n), Ein_{1,n-1})$-structure, is conformally flat. Conversely, by Theorem \ref{Liouville theorem}, any conformally flat spacetime of dimension $n \geq 3$ admits a $(O(2,n), Ein_{1,n-1})$-structure. We deduce the following statement.

\begin{proposition}\label{conformally flat and (G,X)-structure}
A conformally flat Lorentzian structure on a manifold $M$ of dimension $n \geq 3$ is equivalent to a $(O_0(2,n),Ein_{1,n-1})$-structure. \qed
\end{proposition}

From the causal point of view, it is more relevant to consider as model space the universal cover of Einstein universe with its group of conformal diffeomorphisms. As a result, a conformally flat Lorentzian structure on a manifold $M$ of dimension $n \geq 3$ is encoded by the data of a development pair $(D,\rho)$ where $D: \tilde{M} \to \eeu$ is a developing map and $\rho: \pi_1(M) \to \Conf(\eeu)$ is the associated holonomy morphism~\footnote{We direct the reader not familiar with $(G,X)$-structures to \cite[Chapter 5]{goldman}.}.

\begin{lemma}\label{lemma: restriction of D to a causal curve}
The restriction of the developing map $D$ to a causal curve of $\tilde{M}$ is injective.
\end{lemma}

\begin{proof}
Let $\gamma: I \subset \R \to \tilde{M}$ be a causal curve. Since $D$ is conformal, it follows that $D \circ \gamma: I \to \eeu$ is a causal curve. Then, if there exist $t_0, t_1 \in I$ such that $t_0 \neq t_1$ and $D(\gamma(t_0)) = D(\gamma(t_1))$, the curve $D \circ \gamma$ would be a causal loop. Contradiction.
\end{proof}

\begin{definition} \label{def: developable}
A conformally flat spacetime $M$ is said \emph{developable} if any developing map descends to the quotient, giving a local diffeomorphism from $M$ to $\eeu$, called again developing map.
\end{definition}

\begin{lemma} \label{lemma: a developable conf. flat spacetime is strongly causal}
Any developable conformally flat spacetime $M$ is \emph{strongly causal}.
\end{lemma}

\begin{proof}
Let $D: M \to \eeu$ be a developing map. Let $p \in M$ and let $U$ be a neighborhood of $p$. Without loss of generality, we suppose that the restriction of $D$ to $U$ is a diffeomorphism on its image. Then, $D(U)$ is a neighborhood of $D(p)$. Since $\eeu$ is GH, it is in particular strongly causal. Thus, there exists a neighborhooh $V'$ of $D(p)$ contained in $D(U)$ and causally convex in $\eeu$. Let $V$ be the preimage of $V'$ under $D_{|U}$. By definition, $V$ is a neighborhood of $p$ contained in $U$. Moreover, $V$ is causall convex in $M$. Indeed, let $\gamma$ be a causal curve of $M$ joining two points $q, q' \in V$. By Lemma \ref{lemma: restriction of D to a causal curve}, the image under $D$ of $\gamma$ is a causal curve of $\eeu$ joining $D(q)$ to $D(q')$. Since $V'$ is causally convex, $D(\gamma)$ is contained in $V$. If $\gamma$ is not contained in $V$, there exists $r \in \gamma \cap (U \backslash V)$. Hence, $D(r) \in D(\gamma) \backslash V'$. Contradiction.
\end{proof}

\subsection{Construction of an enveloping space} \label{sec: def enveloping space}

Let $V$ be a simply-connected globally hyperbolic conformally flat spacetime of dimension $n \geq 3$. In this section, we prove  Theorem \ref{intro: enveloping space theorem}. Our proof still hold if we weaken the assumption \emph{simply-connected} by \emph{developable}.\\

Let $(D, \varphi)$ be a pair where 
\begin{itemize}
\item $D: V \to \eeu$ is a developing map;
\item $\varphi: \eeu \to \mathbb{S}^{n-1} \times \R$ is a spatio-temporal decomposition of $\eeu$.
\end{itemize}
Throughout this section, we call $\pi$ the projection of $\eeu$ on $\mathbb{S}^{n-1}$ defined as $\pi_0 \circ \varphi$ where $\pi_0: \mathbb{S}^{n-1} \times \R \to \mathbb{S}^{n-1}$ is the projection on the first factor.

\paragraph{Timelike foliation on $V$.} Consider the vector field $T$ on $V$ defined as the pull-back by $D$ of $\partial_t$. The flow of $T$ defines a foliation of $V$ by smooth timelike curves. Let $\mathcal{B}$ be the leaf space, namely the quotient space of $V$ by the equivalence relation that identifies two points if they are on the same leaf. We denote by $\psi: V \to \mathcal{B}$ the canonical projection. 

\begin{fact}\label{lemma: leaf space}
The leaf space $\mathcal{B}$ is homeomorphic to a Cauchy hypersurface $S$ of $V$.
\end{fact}

\begin{proof}
Every leaf is a timelike curve of $V$ and so meets $S$ in a unique point. Therefore, the restriction of $\psi$ to $S$ is a continuous bijection on $\mathcal{B}$. The restriction of $\psi$ to $S$ is open. Indeed, any open subset $U$ of $S$ coincides with $\psi_{|S}^{-1}(\psi_{|S}(U))$, so $\psi_{|S}(U)$ is open in $\mathcal{B}$. Then, the restriction of $\psi$ to $S$ is a homeomorphism on $\mathcal{B}$. 
\end{proof}

By definition, the map $\pi \circ D$ is constant on each leaf, therefore it induces a map $d$ from to $\mathcal{B}$ to $\mathbb{S}^{n-1}$ such that the following diagram commutes:
\[\begin{array}{ccc}
V & \overset{D}{\to} & \eeu  \\
\psi \downarrow & & \downarrow \pi\\
\mathcal{B} & \overset{d}{\to} & \mathbb{S}^{n-1}
\end{array}\]
that is $d \circ \psi = \pi \circ D$. Since $\pi \circ D$ and $\psi$ are submersions, $d$ is a local homeomorphism.

\paragraph{Fiber bundle over the leaf space.} Let $E(V)$ be the fiber bundle over $\mathcal{B}$ defined as the pullback by $d: \mathcal{B} \to \mathbb{S}^{n-1}$ of the trivial bundle $\pi : \eeu \to \mathbb{S}^{n-1}$, in other words:
\begin{align*}
E(V) &:= \{(p,b) \in \eeu \times \mathcal{B}:\ \pi(p) = d(b)\}.
\end{align*}
We denote by $\hat{\pi}: E(V) \to \mathcal{B}$ the projection on the second factor. 

\begin{fact}
The fiber bundle $E(V)$ is trivial. 
\end{fact}

\begin{proof}
Let $f$ be the continuous map from $\mathcal{B} \times \R$ in $E(V)$ that sends $(b,t)$ on $(p,b)$ where $p$ is the point of $\eeu$ with coordinates $(d(b), t)$ in the decomposition $\mathbb{S}^{n-1} \times \R$. It is easy to see that $f$ is bijective. Indeed, the inverse is the continuous map that sends $(p,b) \in E(V)$ on $(b, t) \in \mathcal{B} \times \R$ where $t$ is the projection of $p \in \eeu \simeq \mathbb{S}^{n-1} \times \R$ on~$\R$. Therefore, $f$ is a homeomorphism. Clearly, the following diagram commutes
\[\begin{array}{ccc}
\R \times \mathcal{B} & \overset{f}{\to} & E(V) \\
\hat{\pi}_0 \searrow  &                  & \swarrow \hat{\pi} \\
                      & \mathcal{B}      &
\end{array}
\]
i.e. $\hat{\pi} \circ f = \hat{\pi}_0$. In other words, $f$ is an isomorphism of fiber bundles. The lemma follows.
\end{proof}

The projection on the first factor $\hat{D}: E(V)  \to \eeu$ is a local homeomorphism inducing a structure of conformally flat spacetime on $E(V)$. In particular, $E(V)$ is strongly causal (see Lemma \ref{lemma: a developable conf. flat spacetime is strongly causal}).

\begin{fact}
The fibers of $E(V)$ are inextensible timelike curves.
\end{fact}

\begin{proof}
The fiber $E_b$ of $E(V)$ over a point $b \in \mathcal{B}$ is the set of points $(p,b)$ such that $\pi(p) = d(b)$. It is then easy to see that the restriction of $\hat{D}$ to $E_b$ is a homeomorphism on the fiber $\pi^{-1}(d(b))$ of $\pi: \eeu \to \mathbb{S}^{n-1}$. The lemma follows. 
\end{proof}

\paragraph{Conformal embedding of $V$ in $E(V)$.} Let $i$ be the map from $V$ to $E(V)$ defined by 
\begin{align*}
i(p) &:= (D(p), \psi(p))
\end{align*}
for every $p \in V$.

\begin{fact}
The map $i$ is a conformal embedding of $V$ into $E(V)$.
\end{fact}

\begin{proof}
Since the maps $D$ and $\psi$ are continuous, open and conformal, so does the map $i$. All we need to check is that $i$ is injective. Let $p,q  \in V$ such that $i(p) = i(q)$. Then, $D(p) = D(q)$ and $\psi(p) = \psi(q)$. This last inequality implies that $p$ and $q$ belongs to the same timelike leaf. Since the restriction of $D$ to a leaf is injective (see Lemma \ref{lemma: restriction of D to a causal curve}), it follows from $D(p) = D(q)$ that $p = q$.
\end{proof}

\begin{remark}
The restriction of $\hat{D}$ to $i(M)$ coincide with $D$, more precisely $\hat{D} \circ i = D$.
\end{remark}

Now we prove that the image of $i$ is causally convex in $E(M)$. The proof uses the following lemma.

\begin{lemma}\label{lemma: i(S) disconnects E(M)}
The image under $i$ of a Cauchy hypersurface $S$ of $V$ is a spacelike hypersurface of $E(V)$ which deconnects $E(V)$.
\end{lemma}

\begin{proof}
Since $i$ is a conformal embedding, $i(S)$ is a spacelike embedded hypersurface of $E(V)$. Let $\psi: V \to \mathcal{B}$ be the canonical projection of $V$ on the leaf space $\mathcal{B}$. Recall that the restriction of $\psi$ to $S$ is a homeomorphism on $\mathcal{B}$. Clearly, the map $i_{|S} \circ \psi_{|S}^{-1}: \mathcal{B} \to E(V)$ is a section of $\hat{\pi}$. Hence, $i(S)$ is a global section of $E(V)$. The lemma follows.
\end{proof}

\begin{fact}\label{lemma: M is causally convex in E(M)}
The image $i(V)$ is causally convex in $E(V)$.
\end{fact}

\begin{proof}
Let $p_0, p_1 \in V$ such that there is a causal curve $\gamma$ of $E(V)$ joining $p_0$ to $p_1$. Let $\hat{\gamma}$ be an inextensible causal curve of $E(V)$ containing $\gamma$. Each connected component of the intersection of $\hat{\gamma}$ with $i(V)$ is an inextensible causal curve of $i(V)$. We call $\gamma_0$ and $\gamma_1$ the connected components containing $p_0$ and $p_1$ respectively. To prove that $\gamma$ is contained in $i(V)$ is equivalent to prove that $\gamma_0 = \gamma_1$. Suppose that $\gamma_0$ and $\gamma_1$ are disjoint. Then, each one of them meets $i(S)$ in a single point, $x_0$ and $x_1$ respectively, which are distinct. Therefore, the curve $\hat{\gamma}$ intersects $i(S)$ in at least two distinct points. But, $i(S)$ is acausal in $E(M)$ (see Lemma \ref{lemma: i(S) disconnects E(M)} and \cite[chap. 14, Lemma 45 and Lemma 42]{oneill}). Contradiction.
\end{proof}

\begin{remark}
All the results of this section stated until now are based on the existence of a developing map, and so are still valid if $V$ is not-simply connected but developable.
\end{remark}

\paragraph{Embedding of the conformally flat Cauchy extensions of $V$ in $E(V)$.}

\begin{proposition}\label{propo: conf. flat ext. of V embeds in E(V)}
Let $W$ be a conformally flat Cauchy-extension of $V$. Then, there is a conformal embedding $i'$ of $W$ into $E(V)$ such that the following assertions hold:
\begin{itemize}
\item The image $i'(W)$ is causally convex in $E(V)$ and contains $i(V)$;
\item Every Cauchy hypersurface of $i(V)$ is a Cauchy hypersurface of $i'(W)$.
\end{itemize}
\end{proposition}

The proof of Proposition \ref{propo: conf. flat ext. of V embeds in E(V)} uses the following lemma.

\begin{lemma}\label{lemma: E(V) isomorphic to E(W)}
Let $W$ be a conformally flat Cauchy-extension of $V$. Then, the enveloping spaces $E(V)$ and $E(W)$ are isomorphic.
\end{lemma}

\begin{proof}
Let $f: V \to W$ be a Cauchy embedding and let $D': W \to \eeu$ be a developing map such that $D' \circ f = D$. Consider the foliation of $W$ by inextensible timelike curves induced by the pull-back by $D'$ of the vector field $\partial_t$ on $\eeu \simeq \mathbb{S}^{n-1} \times \R$ and let $\psi': W \to \mathcal{B}'$ be the canonical projection on the leaf space. Since $f$ is a Cauchy-embedding, $f(V)$ is causally convex in $W$ (see \cite[Lemma 8]{Salvemini2013Maximal}). Then, the image under $f$ of every leaf of $V$ is the intersection of a unique leaf of $W$ with $f(V)$. It follows that the map $\psi' \circ f$ descends to the quotient in a diffeomorphism $\bar{f}: \mathcal{B} \to \mathcal{B}'$. Let $d': \mathcal{B}' \to \mathbb{S}^{n-1}$ be the developing map induced by $D'$. It is clear that $d:= d' \circ \bar{f}$ is the developing map induced by $D$. Therefore, the map $F: E(V) \to E(W)$ defined by $F(p,b) = (p, \bar{f}(p))$ is a conformal diffeomorphism that sends every fiber of $E(V)$ on a fiber of $E(W)$. The lemma follows.
\end{proof}

\begin{proof}[Proof of Proposition \ref{propo: conf. flat ext. of V embeds in E(V)}]
Let $j: W \hookrightarrow E(W)$ be the conformal embedding of $W$ into $E(W)$ and let $F: E(V) \to E(W)$ be the isomorphism defined in the proof of Lemma \ref{lemma: E(V) isomorphic to E(W)}. The map $i':= F^{-1} \circ j$ defines a conformal embedding of $W$ into $E(V)$. By Lemma \ref{lemma: M is causally convex in E(M)}, $i(V)$ and $i'(W)$ are causally convex in $E(V)$. Moreover, according to the proof of Lemma \ref{lemma: E(V) isomorphic to E(W)}, the following diagram commutes:
\[\begin{array}{ccc}
V           & \overset{i}{\hookrightarrow} & E(V) \\
f  \downarrow         &                              &  \downarrow F\\
W           & \overset{j}{\hookrightarrow} & E(W)
\end{array}
\]
that is $F \circ i = j \circ f$, i.e.  $ i = F^{-1} \circ j \circ f = i' \circ f$. It follows that:
\begin{itemize}
\item $i(V) = i' (f(V)) \subset i'(W)$.
\item Every Cauchy hypersurface of $i(V)$ is a Cauchy hypersurface of $i'(W)$. Indeed, since $f$ is a Cauchy-embedding, if $S$ is a Cauchy hypersurface of $V$, $f(S)$ is a Cauchy hypersurface of $W$. Then, $i(S) = i'(f(S))$ is a Cauchy hypersurface of $i'(W)$. 
\end{itemize}
The proposition follows.
\end{proof}

\begin{remark}
If $V$ is developable, it is easy to see that any conformally flat Cauchy extension of $W$ is also developable. As a result, Proposition \ref{propo: conf. flat ext. of V embeds in E(V)} is still true in this setting.
\end{remark}

\paragraph{The $\mathcal{C}_0$-maximal extension of $V$.} From now on, we identify $V$ and the conformally flat Cauchy-extensions of $V$ with their images in $E(V)$. Let $S$ be a Cauchy hypersurface of $V$. 

\begin{lemma}\label{lemma: Cauchy development of S}
The Cauchy development of $S$ in $E(V)$ contains all the Cauchy-extensions of $V$. In particular, it contains $V$.
\end{lemma}

\begin{proof}
Let $W$ be a Cauchy extension of $V$. Let $x \in W$ and let $\hat{\gamma}$ be an inextensible causal curve of $E(V)$ going through $x$. Since $W$ is causally convex in $E(V)$ (see Proposition \ref{propo: conf. flat ext. of V embeds in E(V)}), the intersection of $\hat{\gamma}$ with $W$ is an inextensible causal curve $\gamma$ of $W$. By Proposition \ref{propo: conf. flat ext. of V embeds in E(V)}, $S$ is a Cauchy hypersurface of $W$, then $\gamma$ intersects $S$ in a single point. It follows that $x$ belongs to the Cauchy development of $S$ in $E(V)$.
\end{proof}

\begin{proposition}\label{prop: max. ext. dev.}
The Cauchy development $\mathcal{C}(S)$ of $S$ in $E(V)$ is a $\mathcal{C}_0$-maximal extension of~$V$. 
\end{proposition}

\begin{proof}
By \cite[Theorem 38, p. 421]{oneill}, $\mathcal{C}(S)$ is a globally hyperbolic spacetime for which $S$ is a Cauchy hypersurface. According to Lemma \ref{lemma: Cauchy development of S}, it is a Cauchy extension of $V$. Let $W$ be a conformally flat Cauchy-extension of $\mathcal{C}(S)$. In particular, $W$ is a Cauchy-extension of $V$. Then, $W$ embeds conformally in $E(V)$ and the image is a causally convex open subset of $E(V)$ containing $\mathcal{C}(S)$. By Lemma \ref{lemma: Cauchy development of S}, $\mathcal{C}(S)$ is exactly $W$ (seen in $E(V)$). Hence, $\mathcal{C}(S)$ is $\mathcal{C}_0$-maximal.
\end{proof}

\begin{corollary} \label{cor: uniqueness of max. ext. dev.}
The $\mathcal{C}_0$-maximal extension $\mathcal{C}(S)$ is unique up to conformal diffeomorphism. 
\end{corollary}

\begin{proof}
Let $\hat{V}$ another $\mathcal{C}_0$-maximal extension of $V$. By Lemma \ref{lemma: Cauchy development of S}, $\hat{V}$, seen in $E(V)$, is contained in $\mathcal{C}(S)$. If this inclusion is strict, $\mathcal{C}(S)$ would be a Cauchy extension of $\hat{V}$. This contradicts the $\mathcal{C}_0$-maximality of $\hat{V}$.
\end{proof}

We proved again the existence and the uniqueness of the $\mathcal{C}_0$-maximal extension for \emph{simply-connected} conformally flat globally hyperbolic flat spacetimes.

\begin{corollary} \label{cor: Cauchy-compact conformally flat spacetimes}
If $V$ is Cauchy-compact then the $\mathcal{C}_0$-maximal extension of $V$ is conformally equivalent to $\eeu$.
\end{corollary}

\begin{proof}
By Lemma \ref{lemma: leaf space}, the leaf space $\mathcal{B}$ is compact. Since $d$ is a local homeomorphism, it follows that $d$ is a covering. But, $\mathbb{S}^{n-1}$ is simply connected, so $d$ is a homeomorphism. As a result, $\hat{D}: E(V) \to \eeu$ is a conformal diffeomorphism. Indeed, let $p \in \eeu$. There exists a unique $b \in \mathcal{B}$ such that $d(b) = \pi(p)$. Thus, $(p,b)$ is the unique point of $E(V)$ such that $\hat{D}(p,b) = p$. The corollary follows.
\end{proof}

\begin{definition}
The trivial fiber bundle $\hat{\pi}: E(V) \to \mathcal{B}$ is called \textbf{an enveloping space} of $V$.
\end{definition}

The construction of this fiber bundle depends on the choice of a pair $(D, \varphi)$ where
\begin{itemize}
\item $D: V \to \eeu$ is a developing map ;
\item $\varphi: \mathbb{S}^{n-1} \times \R \to \eeu$ is a spatio-temporal decomposition of $\eeu$.
\end{itemize}
Given two such pairs $(D, \varphi)$, $(D', \varphi')$, there exists a conformal transformation $\phi$ of $\eeu$ such that $D' = \phi \circ D$. We say that $(D, \varphi)$ and $(D', \varphi')$ are \emph{equivalent} if $\varphi' = \varphi \circ \phi^{-1}$.

\begin{lemma}
If $(D, \varphi)$ and $(D', \varphi')$ are equivalent, the enveloping spaces $E(V)$ and $E'(V)$ defined by $(D, \varphi)$ and $(D', \varphi')$ are isomorphic, i.e. there exists a conformal diffeomorphism from $E(V)$ to $E'(V)$ which sends fiber on fiber. \qed
\end{lemma}

\begin{example}[Enveloping space of Minkowski spacetime] \label{example: Minkowski}
Minkowski spacetime $\R^{1,n-1}$ is conformally equivalent to the set of points of $\eeu$ which are not causally related to a point $p \in \eeu$. Given a spatio-temporal decomposition $\varphi: \eeu \to \mathbb{S}^{n-1} \times \R$, the enveloping space of $\R^{1,n-1}$ is the complement in $\eeu$ of the fiber going through $p$ of the trivial bundle $\pi_0 \circ \varphi: \eeu \to \mathbb{S}^{n-1}$. This is a trivial fiber bundle over $\mathbb{S}^{n-1} \backslash \{\pi_0 \circ \varphi(p)\} \simeq \mathbb{R}^{n-1}$.
\end{example}

\begin{example}[Enveloping space of De-Sitter spacetime]
Given a spatio-temporal decomposition $\varphi: \eeu \to \mathbb{S}^{n-1} \times \R$, De-Sitter spacetime is conformally equivalent to $\mathbb{S}^{n-1} \times ]0,\pi[$. Therefore, the enveloping space of De-Sitter spacetime relatively to this decomposition is $\pi_0 \circ \varphi: \eeu \to \mathbb{S}^{n-1}$ where $\pi_0: \mathbb{S}^{n-1} \times \R \to \mathbb{S}^{n-1}$ is the projection on the first factor.
\end{example}

\subsection{Causally convex GH open subsets of the enveloping space} \label{sec: Causally convex subsets of E(M)}

In this section, we describe causally convex open subsets of the enveloping space $E(V)$ - defined in the previous section - which are globally hyperbolic. Notice that since $E(V)$ is, \emph{a priori}, not globally hyperbolic (see Example \ref{example: Minkowski}), causal convexity does not imply global hyperbolicity anymore.

By \cite[Theorem 10]{Salvemini2013Maximal}, causally convex GH open subsets of $E(V)$ which contain \emph{conjugate points}, i.e. points whose image under $\hat{D}$ are conjugate in $\eeu$, are conformally equivalent to causally convex open subsets of $\eeu$. These are described in Section~\ref{sec: causally convex open subsets of Ein}. This is why we only deal here with causally convex open subsets of $E(V)$ without conjugate points. We basically generalise the description of causally convex open subsets of $\eeu$.\\ \\ The following definition introduces a class of sections of the enveloping space $E(V) \to \mathcal{B}$, expressed in a global trivialization. 

\begin{definition}
A real-valued function $f$ defined on an open subset $U$ of $\mathcal{B}$ is said \emph{$1$-Lipschitz} if for every $x \in U$, there exists an open neighborhood $U_x$ of $x$ contained in $U$ such that the following hold:
\begin{enumerate}
\item the restriction of $d$ to $U_x$ is injective;
\item the map $f \circ d_{|U_x}^{-1}: d(U_x) \subset \mathbb{S}^{n-1} \to \R$ is $1$-Lipschitz.
\end{enumerate}
\end{definition}

This definition generalizes the (local) notion of $1$-Lipschitz functions defined on an open subset of the sphere $\mathbb{S}^{n-1}$. 

\begin{remark}
The function $f$ is $1$-Lipschitz if and only if its graph is \emph{locally achronal}.
\end{remark}

\begin{lemma}
Any $1$-Lipschitz real-valued function $f$ defined on an open subset $U \subsetneq \mathcal{B}$ admits a unique extension to $\overline{U}$. \qed
\end{lemma}

\begin{proposition}
Any causally convex GH open subset $\Omega$ of $E(V)$ is the domain bounded by the graphs of two $1$-Lipschitz real-valued functions $f^+$ and $f^-$ defined on an open subset $U$ of $\mathcal{B}$ such that:
\begin{enumerate}
\item $f^- < f^+$ on $U$;
\item the extensions of $f^+$ and $f^-$ to $\partial U$ coincide.
\end{enumerate}
\end{proposition}

\begin{proof}
Let $\hat{\pi}$ the natural projection of $E(V)$ on $\mathcal{B}$. We call $U$ the projection of $\Omega$ on $\mathcal{B}$. Since $\Omega$ is causally convex, the intersection of $\Omega$ with any fiber $\hat{\pi}^{-1}(x)$, where $x \in U$, is connected, i.e. it is a segment $\{x\} \times ]f^-(x), f^+(x)[$. Notice that $-\infty < f^-(x)$ and $f^+(x) < + \infty$, otherwise $\Omega$ would contain conjugate points. Contradiction. This defines two real-valued functions $f^+, f^-$ defined on $U$ such that:
\begin{align*}
\Omega &= \{(x,t) \in U \times \R:\ f^-(x) < t < f^+(x)\}.
\end{align*}
Let $x \in U$. Set $p_+ = (x, f^+(x))$. Let $p \in I^-(p^+) \cap \Omega$. Since $\Omega$ is GH, the restriction of $\hat{D}$ to $I^+(p, \Omega)$ is injective and its image is causally convex in $\eeu$ (see \cite[Prop. 2.7 and Cor. 2.8, p.151]{salveminithesis}).  It follows that:
\begin{itemize}
\item $\hat{D}(I^+(p, \Omega))$ is the domain bounded by the graphs of two $1$-Lipschitz real-valued functions $g^- < g^+$ defined on an open subset of $\mathbb{S}^{n-1}$ (see Proposition \ref{lemma: causally convex in ein 1});
\item $d$ is injective on $U_x := \hat{\pi}(I^+(p,\Omega))$ and $g^+ = f^+ \circ d_{|U_x}^{-1}$.
\end{itemize}
Thus, $f^+$ is $1$-Lipschitz. The proof is similar for $f^-$ with the reverse time-orientation.
Lemma \ref{lemma: the closure of a causally convex is causally convex} is still valid for causally convex open subsets of $E(V)$. Hence, the same arguments used in the proof of Proposition \ref{lemma: causally convex in ein 1} shows that the extensions of $f^+$ and $f^-$ to $\partial U$ coincide.
\end{proof}

\begin{corollary}
The graphs of $f^+$ and $f^-$ are achronal.
\end{corollary}

\begin{proof}
The proof is similar than that of Fact 2 in the proof of Proposition \ref{lemma: causally convex in ein 1}.
\end{proof}

\section{$\mathcal{C}_0$-maximal extensions of conformally flat globally hyperbolic spacetimes} \label{sec: maximal ext.}

In this section, we give a new proof of the existence and the uniqueness of the $\mathcal{C}_0$-maximal extension of a globally hyperbolic conformally flat spacetime $V$ of dimension $n \geq 3$, using the notion of enveloping space introduced in Section \ref{sec: enveloping space}.

When $V$ is simply-connected, the proof is given in Section \ref{sec: def enveloping space} (see Proposition \ref{prop: max. ext. dev.} and Corollary \ref{cor: uniqueness of max. ext. dev.}). We deal here with the case where $V$ is not simply-connected. The proof consists to extend the action of $\pi_1 V$ on $\tilde{V}$ to a proper action on the $\mathcal{C}_0$-maximal extension of $\tilde{V}$. We prove then that the $\mathcal{C}_0$-maximal extension of $V$ is the quotient of the $\mathcal{C}_0$-maximal extension of $\tilde{V}$ by $\pi_1 V$. 

\paragraph{Notations.} Let $D: \tilde{V} \to \eeu$ be a developing map and let $\rho: \Gamma \to \Conf(\eeu)$ be the associated holonomy representation, where $\Gamma := \pi_1(V)$. 

Fix a decomposition $\varphi: \eeu \to \mathbb{S}^{n-1} \times \R$. Let $E(\tilde{V})$ be the enveloping space related to the pair $(D, \varphi)$ (see Section \ref{sec: def enveloping space}). We denote by $\hat{\pi}: E(\tilde{V}) \to \mathcal{B}$ the projection on the second factor and by $\hat{D}: E(\tilde{V}) \to \eeu$ the projection on the first factor.

\begin{fact}
If $\tilde{V}$ is Cauchy-compact, the $\mathcal{C}_0$-maximal extension of $V$ is a finite quotient of $\eeu$.
\end{fact}

\begin{proof}
This is an immediate consequence of Corollary \ref{cor: Cauchy-compact conformally flat spacetimes}.
\end{proof}

Suppose now that $\tilde{V}$ is not Cauchy-compact. Let $\tilde{S}$ be a non-compact Cauchy hypersurface of $\tilde{V}$. The $\mathcal{C}_0$-maximal extension of $\tilde{V}$ is the Cauchy development $\mathcal{C}(\tilde{S})$ of $\tilde{S}$ in $E(\tilde{V})$ (see Proposition \ref{prop: max. ext. dev.}).  In what follows, we extend the action of $\Gamma$ to $\mathcal{C}(\tilde{S})$.

\paragraph{Action of $\Gamma$ on $\mathcal{C}(\tilde{S})$.} The points of $\mathcal{C}(\tilde{S})$ are characterized by their shadows on $\tilde{S}$ (see Proposition \ref{prop: shadows}). As a result, we show that the action of $\Gamma$ on $\tilde{S}$ induces naturally an action of $\Gamma$ on $\mathcal{C}(\tilde{S})$: 

\begin{proposition}\label{prop: action of the Cauchy dev. of S}
Let $p \in \mathcal{C}(\tilde{S})$ and let $\gamma \in \Gamma$. There exists a unique point $\gamma.p$ in $\mathcal{C}(\tilde{S})$ such that its shadow on $\tilde{S}$ is exactly the image under $\gamma$ of the shadow of $p$ on $\tilde{S}$. This defines an action of $\Gamma$ on $\mathcal{C}(\tilde{S})$ which satisfies the following properties:
\begin{itemize}
\item the restriction of the action of $\Gamma$ on $\mathcal{C}(\tilde{S})$ to $\tilde{V}$ coincide with the usual action of $\Gamma$ on $\tilde{V}$;
\item the action of $\Gamma$ on $\mathcal{C}(\tilde{S})$ preserves the causality relations, i.e. for every $p \in \mathcal{C}(\tilde{S})$ and for every $\gamma \in \Gamma$, we have:
\begin{align*}
p \in J^-(q) \Leftrightarrow \gamma.p \in J^-(\gamma.q)\\
p \in I^-(q) \Leftrightarrow \gamma.p \in I^-(\gamma.q);
\end{align*}
\item the restriction of $\hat{D}$ to $\mathcal{C}(\tilde{S})$ is $\rho$-equivariant, i.e. for every $p \in \mathcal{C}(\tilde{S})$ and for every $\gamma \in \Gamma$, we have:
\begin{align*}
\hat{D}(\gamma.p) &= \rho(\gamma) \hat{D}(p).
\end{align*}
\end{itemize}
\end{proposition}

We prove Proposition \ref{prop: action of the Cauchy dev. of S} by an analysis-synthesis reasoning. In the analysis, we suppose that $\gamma.p$ exists and is unique, and we look for properties satified by $\gamma.p$ that will characterize it. In the synthesis, we use the criteria found in the analysis to determine the point $\gamma.p$. In what follows, the shadow of a point $p \in \mathcal{C}(\tilde{S})$ on $\tilde{S}$ is denoted by $O(p)$. 

\subparagraph{Analysis.} Suppose that for every $\gamma \in \Gamma$ and every $p \in \mathcal{C}(\tilde{S})$, the point $\gamma.p$ exists and is unique. For every $\gamma \in \Gamma$ and every $p \in \tilde{V}$, we denote by $\gamma p$ (without the dot between $\gamma$ and $p$) the usual action of the deck transformation $\gamma$ on $p$. Let us start with this easy remark.

\begin{remark}\label{remark: the action on the Cauchy dev is an extension of the action on V univ}
Let $p \in \mathcal{C}(\tilde{S})$ and let $\gamma \in \Gamma$. If $p \in \tilde{V}$, then $\gamma.p = \gamma p$. Indeed, since the action of $\Gamma$ on $\tilde{V}$ respect the causality relations, we have $J^-(\gamma p) = \gamma J^-(p)$. Therefore,
\begin{align*}
O(\gamma p) := J^-(\gamma p) \cap \tilde{S} = \gamma J^-(p) \cap \tilde{S} = \gamma (J^-(p) \cap \gamma^{-1}\tilde{S}).
\end{align*}
Since $\tilde{S}$ is $\Gamma$-invariant, we deduce that $\gamma (J^-(p) \cap \gamma^{-1}\tilde{S}) = \gamma O(p)$. Hence, $O(\gamma p) = \gamma O(p)$, i.e. $\gamma.p = \gamma p$.
\end{remark}

\begin{lemma}\label{lemma: the action on the Cauchy dev. respect causality}
The map which associates to every $(\gamma, p) \in \Gamma \times \mathcal{C}(\tilde{S})$ the point $\gamma.p \in \mathcal{C}(\tilde{S})$ is a group action. Moreover, this action respect causality relations.

\end{lemma}

\begin{proof}
The fact that the map $(\gamma, p) \in \Gamma \times \mathcal{C}(\tilde{S}) \mapsto \gamma.p$ is a group action follows easily from the fact that the restriction to $\Gamma \times \tilde{S}$ is the usual group action of $\Gamma$ on $\tilde{S}$.

Let $\gamma \in \Gamma$ and let $p,q \in \mathcal{C}(\tilde{S})$. Suppose that $p,q \in J^+(S)$. Then,
\begin{align*}
p \in J^-(q) \Leftrightarrow O(p) \subset O(q) \Leftrightarrow O(\gamma.p) := \gamma O(p) \subset \gamma O(q) =: O(\gamma.q) \Leftrightarrow \gamma.p \in J^-(\gamma.q).
\end{align*}
By symmetry, the same arguments still hold if $p,q \in J^-(\tilde{S})$. It remains the case where $q \in I^+(\tilde{S})$ and $p \in I^-(\tilde{S})$. In this case, we have
\begin{align*}
p \in J^-(q) \Leftrightarrow O(p) \cap O(q) \not = \emptyset \Leftrightarrow \gamma O(p) \cap \gamma O(q) \not = \emptyset \Leftrightarrow O(\gamma.p) \cap O(\gamma.q) \not = \emptyset.
\end{align*}
Since $\Gamma$ preserves $\tilde{S}$, it preserves the chronological future/past of $\tilde{S}$ in $\mathcal{C}(\tilde{S})$. Thus, $O(\gamma.p) \cap O(\gamma.q) \not = \emptyset \Leftrightarrow \gamma.p \in J^-(\gamma.q)$. The lemma follows.
\end{proof}

\begin{lemma}\label{lemma: the restriction of dev to the max. ext. is equivariant}
The restriction of the developing map $\hat{D}: E(\tilde{V}) \to \eeu$ to $\mathcal{C}(\tilde{S})$ is $\rho$-equivariant, i.e. for every $p \in \mathcal{C}(\tilde{S})$ and for every $\gamma \in \Gamma$, we have $\hat{D}(\gamma.p) = \rho(\gamma)\hat{D}(p)$.
\end{lemma}

\begin{proof}
Let $\gamma \in \Gamma$ and let $p \in \mathcal{C}(\tilde{S})$. Suppose that $p \in I^+(\tilde{S})$. Let $q \in \mathcal{C}(\tilde{S}) \cap I^+(p)$. Since $\mathcal{C}(\tilde{S})$ is globally hyperbolic, the restriction of $\hat{D}$ to $I^-(q, \mathcal{C}(\tilde{S}))$ is injective and its image is causally convex in $\eeu$ (see \cite[Prop. 2.7 and Cor. 2.8, p. 151]{salveminithesis}). It follows that $D(I^-(q) \cap \tilde{S})$ is an achronal hypersurface of $\hat{D}(I^-(q, \mathcal{C}(\tilde{S})))$. 

Set $\Sigma:= D(I^-(q) \cap \tilde{S})$. Since $O(p) \subset I^-(q, \mathcal{C}(\tilde{S}))$, the restriction of $D$ to $O(p)$ is injective and its image is exactly $O(\hat{D}(p), \Sigma)$. Thus
\begin{align}\label{equality 1: proof lemma dev equivariant}
D(\gamma O(p)) = \rho(\gamma) D(O(p))
               = \rho(\gamma) O(\hat{D}(p), \Sigma)
               = O (\rho(\gamma) \hat{D}(p), \rho(\gamma) \Sigma).
\end{align}
By definition, $\gamma O(p) = O(\gamma. p)$. From Lemma \ref{lemma: the action on the Cauchy dev. respect causality}, we get $O(\gamma.p) \subset I^-(\gamma.q, \mathcal{C}(\tilde{S}))$. As above, we deduce that the restriction of $D$ to $O(\gamma.p)$ is injective and its image is exactly $O(\hat{D}(\gamma.p), D(I^-(\gamma.p) \cap \tilde{S}))$. But, $D(I^-(\gamma.p) \cap \tilde{S}) = \rho(\gamma) \Sigma$. Indeed, since $O(\gamma.q) = \gamma O(q)$ we have $I^-(\gamma.q) \cap \tilde{S} = \gamma (I^-(q) \cap \tilde{S})$. Thus, $D(I^-(\gamma.q) \cap \tilde{S}) = \rho(\gamma) \Sigma$. Hence,
\begin{align}\label{equality 2: proof lemma dev equivariant}
D(\gamma O(p)) = D(O(\gamma.p)) = O(\hat{D}(\gamma.p), \rho(\gamma) \Sigma).
\end{align}
It follows from (\ref{equality 1: proof lemma dev equivariant}) and (\ref{equality 2: proof lemma dev equivariant}) that $O (\rho(\gamma) \hat{D}(p), \rho(\gamma) \Sigma) = O(\hat{D}(\gamma.p), \rho(\gamma) \Sigma)$. By Proposition \ref{prop: shadows}, we deduce that $\hat{D}(\gamma.p) = \rho(\gamma) \hat{D}(p)$.

If $p \in I^-(\tilde{S})$, the same arguments hold with the reverse time-orientation. Lastly, if $p \in \tilde{S}$, the lemma follows from Remark \ref{remark: the action on the Cauchy dev is an extension of the action on V univ}.
\end{proof}

\begin{remark}\label{remark: appropriate neighborhood}
In the proof of Lemma \ref{lemma: the restriction of dev to the max. ext. is equivariant}, we can replace $I^-(q, \mathcal{C}(\tilde{S}))$ by any other causally convex open neighborhood $U$ of $p$ in $\mathcal{C}(\tilde{S})$ such that:
\begin{itemize}
\item $U$ intersects $\tilde{S}$;
\item the restriction of $\hat{D}$ to $U$ is injective;
\item the image $\hat{D}(U)$ is causally convex in $\eeu$.
\end{itemize}
\end{remark}

\subparagraph{Synthesis.} Let $p \in \mathcal{C}(\tilde{S})$ and let $\gamma \in \Gamma$. Without loss of generality, we suppose that $p \in I^+(\tilde{S})$. The idea is to reconstruct the proof of Lemma \ref{lemma: the restriction of dev to the max. ext. is equivariant} to determine $\gamma.p$. More precisely, we choose a relevant causally convex open neighborhood $U$ of $p$ in $\mathcal{C}(\tilde{S})$ satisfying the properties stated in Remark \ref{remark: appropriate neighborhood} and we construct the open subset that will turn out to be $\gamma.U$ and will therefore contain $\gamma.p$ (see Figure \ref{figure: maximal extension}).\\

Fix $q \in \mathcal{C}(\tilde{S}) \cap I^+(p)$. Let $U$ be the Cauchy development of $I^-(q) \cap \tilde{S}$ in $E(\tilde{V})$. 

\begin{lemma}\label{lemma: neighborhood of p in eeu}
The image under $D$ of $I^-(q) \cap \tilde{S}$ is achronal in $\eeu$.
\end{lemma}

\begin{proof}
We have $I^-(q) \cap \tilde{S} \subset I^-(q, \mathcal{C}(\tilde{S}))$. Since $\mathcal{C}(\tilde{S})$ is GH, the restriction of $\hat{D}$ to $I^-(q, \mathcal{C}(\tilde{S}))$ is injective (see \cite[Prop. 2.7, p. 151]{salveminithesis}). Then, $D(I^-(p) \cap \tilde{S})$ is achronal in $\hat{D}(I^-(q, \mathcal{C}(\tilde{S})))$. This last one being causally convex in $\eeu$ (see \cite[Cor. 2.8, p. 151]{salveminithesis}), we deduce easily that $D(I^-(p) \cap \tilde{S})$ is achronal in $\eeu$.
\end{proof}

\begin{fact}\label{lemma: neighborhood of p}
The restriction of $\hat{D}$ to $U$ is injective. Moreover, the image under $\hat{D}$ of $U$ is equal to the Cauchy development of $D(I^-(q) \cap \tilde{S})$ in $\eeu$.
\end{fact}

\begin{proof}
Clearly, $U \subset \mathcal{C}(\tilde{S}) \cap I^-(q)$. Since $\mathcal{C}(\tilde{S})$ is causally convex in $E(\tilde{V})$, we have $\mathcal{C}(\tilde{S}) \cap I^-(q) = I^-(q, \mathcal{C}(\tilde{S}))$. Since the restriction of $\hat{D}$ to $I^-(q, \mathcal{C}(\tilde{S}))$ is injective and its image is causally convex in $\eeu$ (see \cite[Prop. 2.7 and Cor. 2.8, p. 151]{salveminithesis}), the lemma follows.
\end{proof}

Let $W$ be the Cauchy development of $\gamma (I^-(q) \cap \tilde{S})$ in $E(\tilde{V})$. 

\begin{remark}\label{remark: neighborhood of gamma.p}
The image under $D$ of $\gamma (I^-(q) \cap \tilde{S})$ is achronal in $\eeu$. Indeed, $D(\gamma (I^-(q) \cap \tilde{S})) = \rho(\gamma) D(I^-(q) \cap \tilde{S})$. The assertion follows then from Lemma \ref{lemma: neighborhood of p in eeu}.
\end{remark}

\begin{fact}\label{lemma: neighborhood of gamma.p}
The restriction of $\hat{D}$ to $W$ is injective. Moreover, the image under $\hat{D}$ of $W$ is equal to the Cauchy development of $D(\gamma (I^-(q) \cap \tilde{S}))$ in $\eeu$.
\end{fact}

\begin{proof}
Let $r, r' \in W$ such that $\hat{D}(r) = \hat{D}(r')$. Let $E_{\hat{\pi}(r)}$ and $E_{\hat{\pi}(r')}$ be the fibers going through $r$ and $r'$. By definition, their images under $\hat{D}$ are the fibers of the trivial bundle $\eeu \to \mathbb{S}^{n-1}$ going through $\hat{D}(r)$ and $\hat{D}(r')$. Since $\hat{D}(r) = \hat{D}(r')$, these fibers coincide. We call this fiber $\Delta$.  Since$E_{\hat{\pi}(r)}$ and $E_{\hat{\pi}(r')}$ are inextensible timelike curves, they intersect $\gamma (I^-(q) \cap \tilde{S}$ in two points $r_0$ and $r_0'$. Hence, $D(r_0), D(r_0') \in \Delta \cap D(\gamma (I^-(q) \cap \tilde{S}))$. But, $D(\gamma (I^-(q) \cap \tilde{S}))$ is achronal in $\eeu$ (see Remark \ref{remark: neighborhood of gamma.p}). Thus, the timelike line $\Delta$ intersects $D(\gamma (I^-(q) \cap \tilde{S}))$ at most once. Hence, $D(r_0) = D(r_0')$. The restriction of $D$ to $\gamma (I^-(q) \cap \tilde{S})$ being injective, we deduce that $r_0 = r_0'$. Thus, $E_{\hat{\pi}(r)} = E_{\hat{\pi}(r')}$. Since the restriction of $\hat{D}$ to any causal curve is injective (see Lemma \ref{lemma: restriction of D to a causal curve}), we get $r = r'$.

We deduce easily from the injectivity of $\hat{D}$ on $W$ that $\hat{D}(W)$ is contained in the Cauchy development of $D(\gamma (I^-(q) \cap \tilde{S}))$ in $\eeu$. This inclusion is clearly a Cauchy-embedding. Since $W$ is $\mathcal{C}_0$-maximal (see Proposition \ref{prop: max. ext. dev.}), we deduce equality. 
\end{proof}

\begin{proof}[Proof of Proposition \ref{prop: action of the Cauchy dev. of S}]
Set $\Sigma := D(I^-(q) \cap \tilde{S})$. Since $O(p) \subset U$, it follows from Fact \ref{lemma: neighborhood of p} that $D(O(p))$ is equal to the shadow $O(\hat{D}(p), \Sigma)$ of $\hat{D}(p)$ on $\Sigma$. Hence
\begin{align*}
D(\gamma O(p)) = \rho(\gamma) D(O(p)) = \rho(\gamma) O(\hat{D}(p), \Sigma) = O(\rho(\gamma) \hat{D}(p), \rho(\gamma) \Sigma).
\end{align*}
Hence, $O(\rho(\gamma) \hat{D}(p), \rho(\gamma) \Sigma) \subset \hat{D}(W)$. Since $\hat{D}(W)$ is the Cauchy development of $\rho(\gamma) \Sigma$ (see Fact \ref{lemma: neighborhood of gamma.p}), it follows that $\rho(\gamma) \hat{D}(p) \in \hat{D}(W)$. Thus, there exists a unique point $p' \in U$ such that $\hat{D}(p') = \rho(\gamma) \hat{D}(p)$. It is clear that $O(p') = \gamma O(p)$. If there is another point $p'' \in \mathcal{C}(\tilde{S})$ such that $O(p'') = \gamma O(p)$, by Proposition \ref{prop: shadows}, we get $p' = p''$. Then, we set $\gamma.p := p'$. 
\end{proof}

\begin{figure}[h!]
\centering
\includegraphics[scale=1]{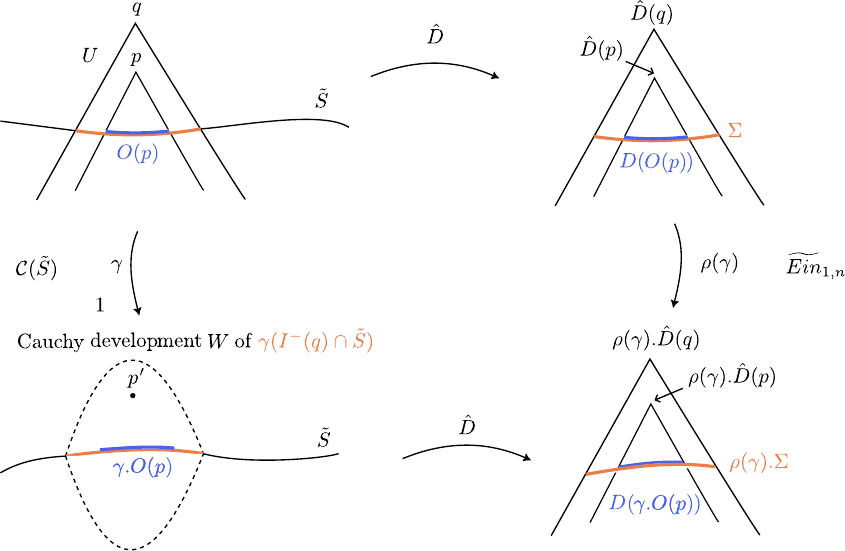}
\caption{Action of $\Gamma$ on the Cauchy development of $\tilde{S}$ in the enveloping space $E(\tilde{S})$.}
\label{figure: maximal extension}
\end{figure}

\paragraph{Dynamical properties of the action of $\Gamma$ on $\mathcal{C}(\tilde{S})$.} In this paragraph, we prove that the action of $\Gamma$ on $\mathcal{C}(\tilde{S})$ is free and properly discontinuous.

\begin{lemma}\label{lemma: the shadows are topological disks}
Let $p$ be a point in the complement of $\tilde{S}$ in $\mathcal{C}(\tilde{S})$. Then, the shadow of $p$ on $\tilde{S}$ is a topological disk of dimension~$n - 1$.
\end{lemma}

\begin{proof}
We can suppose without loss of generality that $p \in I^+(\tilde{S})$. Let $q \in I^+(p)$ and set $\Sigma := D(I^-(q) \cap \tilde{S})$. We prove that $\hat{D}(I^-(q, \mathcal{C}(\tilde{S})))$ is contained in the affine chart $\Mink_-(\hat{D}(q))$ (see Section \ref{sec: euu}). Suppose there exists $r \in I^-(q, \mathcal{C}(\tilde{S}))$ such that \mbox{$\hat{D}(r) \not \in \Mink_-(\hat{D}(q))$}. Then, $I(\hat{D}(q), \hat{D}(r))$ contains conjugate points. Since the image under $\hat{D}$ of $I^-(q, \mathcal{C}(\tilde{S}))$ is causally convex in $\eeu$ (see \cite[Cor. 2.8, p. 151]{salveminithesis}), we deduce that $I^-(q, \mathcal{C}(\tilde{S})$ admits a photon whose image under $\hat{D}$ contains conjugate points. Therefore, by \cite[Theorem 10]{Salvemini2013Maximal}, $\mathcal{C}(\tilde{S})$ is conformally equivalent to $\eeu$. Contradiction. Then, $\hat{D}(p)$ and $\Sigma$ are contained in $\Mink_-(\hat{D}(q))$. As a result, the shadow $O(\hat{D}(p), \Sigma)$ is the intersection in Minkowski spacetime of the past causal cone of $\hat{D}(p)$ with $\Sigma$. Thus, the map which associates to every past causal direction at $\hat{D}(p)$, the intersection of the straight line tangent to this direction with $\Sigma$ is a homeomorphism. Hence, $O(\hat{D}(p), \Sigma)$ is a topological $(n-1)$-disk. Since the restriction of $\hat{D}$ to $O(p)$ is a diffeomorphism on $O(\hat{D}(p), \Sigma)$ (see \cite[prop. 2.7, p. 151]{salveminithesis}), the lemma follows.  
\end{proof}

\begin{proposition}\label{prop: properness of the action}
The action of $\Gamma$ on $\mathcal{C}(\tilde{S})$ is free and properly discontinous.
\end{proposition}

\begin{proof}
Let $p \in \mathcal{C}(\tilde{S})$ and let $\gamma \in \Gamma$ such that $\gamma. p = p$.  Thus, $\gamma$ preserves $O(p)$. Since $O(p)$ is a topological disk (see Lemma \ref{lemma: the shadows are topological disks}), by Brouwer's theorem, $\gamma$ admits a fixed point in $O(p)$. Since the action of $\Gamma$ on $\tilde{S}$ is free, we deduce that $\gamma = id$. This proves that the action of $\Gamma$ on $\mathcal{C}(\tilde{S})$ is free.

Suppose the action of $\Gamma$ on $\mathcal{C}(\tilde{S})$ is not properly discontinuous. Then, by \cite[Proposition 1]{frances2005lorentzian}, there exist a sequence $\{p_i\}$ of points of $\mathcal{C}(\tilde{S})$ converging to some point $p_{\infty} \in \mathcal{C}(\tilde{S})$ and a divergent sequence $\{\gamma_i\}$ of elements of $\Gamma$ such that $\{\gamma_i.p_i\}$ converges to some point $q_{\infty} \in \mathcal{C}(\tilde{S})$. Without loss of generality, we can suppose that $p_{\infty}, q_{\infty} \in J^+(\tilde{S})$.

Let $p \in I^+(p_{\infty})$. Since $\lim p_i = p_{\infty}$, all the $p_i$ belong to $I^-(p)$ except a finite number. Hence, $O(p_i) \subset O(p)$ for every $i \geq i_0$ where $i_0$ is a natural integer. Let $x_i \in O(p_i)$. Up to extracting, $\{x_i\}$ converges to some $x_{\infty} \in O(p)$. Similarly, let $q \in I^+(q_{\infty})$; all the $\gamma_i p_i$ belong to $O(q)$ except a finite number. Hence, $\gamma_i O(p_i) = O(\gamma_i p_i) \subset O(q)$. Therefore, up to extracting, $\{\gamma_i x_i\}$ converges to some point $y_{\infty} \in O(q)$. This contradicts the properness of the action of $\Gamma$ on $\tilde{S}$ (see \cite[Proposition 1]{frances2005lorentzian}). 
\end{proof}

\paragraph{The $\mathcal{C}_0$-maximal extension of $V$.}

\begin{proposition}\label{prop: max. ext. of V}
The quotient space $\Gamma \backslash \mathcal{C}(\tilde{S})$ is a $\mathcal{C}_0$-maximal extension of $V$.
\end{proposition}

\begin{proof}
Let $i: \tilde{V} to E(\tilde{V})$ the embedding of $\tilde{V}$ in $E(\tilde{V})$ defined in Section \ref{sec: enveloping space}. The co-restriction of $i$ to $\mathcal{C}(\tilde{S})$ is a Cauchy embedding, denoted by $\tilde{f}$. Let $\pi' : \mathcal{C}(\tilde{S}) \to \Gamma \backslash \mathcal{C}(\tilde{S})$ be the canonical projection. Since the action of $\Gamma$ on $i(\tilde{V})$ coincide with the usual action of $\Gamma$ on $\tilde{V}$, the map $\tilde{f}$ descends to quotient in a Cauchy embedding $f: V \to \Gamma \backslash \mathcal{C}(\tilde{S})$. It is easy to see that since $\mathcal{C}(\tilde{S})$ is $\mathcal{C}_0$-maximal (see Proposition \ref{prop: max. ext. dev.}), the quotient $\Gamma \backslash \mathcal{C}(\tilde{S})$ is also $\mathcal{C}_0$-maximal. The proposition follows.
\end{proof}

\begin{corollary}
The $\mathcal{C}_0$-maximal extension $\Gamma \backslash \mathcal{C}(\tilde{S})$ is unique up to conformal diffeomorphism.
\end{corollary}

\begin{proof}
Let $f: V \to W$ be a Cauchy embedding of $V$ in a $\mathcal{C}_0$-maximal globally hyperbolic conformally flat spacetime $W$. It lifts to a Cauchy embedding $\tilde{f}: \tilde{V} \to \tilde{W}$. By Proposition \ref{propo: conf. flat ext. of V embeds in E(V)}, $\tilde{W}$ admits a conformal copy in $E(\tilde{V})$ contained in $\mathcal{C}(\tilde{S})$, where $\tilde{S}$ is a Cauchy hypersurface of $\tilde{V}$ (seen in $E(\tilde{V})$). The inclusion of $\tilde{W}$ in $\mathcal{C}(\tilde{S})$ is a Cauchy embedding which descends to the quotient in a Cauchy embedding from $W$ to $\Gamma \backslash \mathcal{C}(\tilde{S})$. By Proposition \ref{prop: max. ext. of V}, this last one is surjective. The corollary follows.
\end{proof}

The proof of Proposition \ref{prop: max. ext. of V} is based on the fact that if the universal covering of a globally hyperbolic spacetime is maximal then this spacetime is maximal. \emph{A priori}, the converse assertion is not true in general. However, Proposition \ref{prop: max. ext. of V} allows to prove that it is true in the conformally flat setting.

\begin{corollary} \label{cor: maximality of the universal covering}
Let $V$ be a $\mathcal{C}_0$-maximal spacetime. Then, the universal covering of $V$ is $\mathcal{C}_0$-maximal.
\end{corollary}

\begin{proof}
Let $\tilde{S}$ be a Cauchy hypersurface of $\tilde{V}$. By Proposition \ref{prop: max. ext. of V}, there is a Cauchy embedding $f$ from $V$ to $\Gamma \backslash \mathcal{C}(\tilde{S})$. This last one lifts to a Cauchy embedding $\tilde{f}: \tilde{V} \to \mathcal{C}(\tilde{S})$. Since $V$ is maximal, $f$ is surjective. Therefore, $\tilde{f}$ is surjective. Since $\mathcal{C}(\tilde{S})$ is $\mathcal{C}_0$-maximal (see Proposition \ref{prop: max. ext. dev.}), $\tilde{V}$ is $\mathcal{C}_0$-maximal.
\end{proof}

\section{$\mathcal{C}_0$-maximal extensions respect inclusion} \label{sec: Maximal extensions respect inclusion}

In this section, we show that the \emph{functor maximal extension} respect inclusion in the setting of conformally flat spacetimes. More precisely, we establish the following result:

\begin{theorem}\label{main theorem}
Let $V$ be a conformally flat globally hyperbolic spacetime and let $U$ be a causally convex open subset of $V$. Then, there is a conformal embedding from the $\mathcal{C}_0$-maximal extension $\hat{U}$ of $U$ into the $\mathcal{C}_0$-maximal extension $\hat{V}$ of $V$. Moreover, the image of this embedding is causally convex in $\hat{V}$.
\end{theorem}

We first prove this result in the case where $V$ is simply-connected in Section \ref{sec: developable case} before dealing with the general case in Section \ref{sec: general case}. 

\subsection{The simply-connected case} \label{sec: developable case}

We prove Theorem \ref{main theorem} in the case where $V$ is simply-connected:

\begin{proposition}\label{prop: max. ext. respect inclusion dev. case}
Let $V$ be a \textbf{simply-connected} conformally flat globally hyperbolic spacetime and let $U$ be a causally convex open subset of $V$. Then, there is a conformal embedding from the $\mathcal{C}_0$-maximal extension $\hat{U}$ of $U$ into the $\mathcal{C}_0$-maximal extension $\hat{V}$ of $V$. Moreover, the image is causally convex in $\hat{V}$.
\end{proposition}

The key idea is to realize the $\mathcal{C}_0$-maximal extensions of $U$ and $V$ in the enveloping space $E(V)$ so we can compare them. The proof of Proposition \ref{prop: max. ext. respect inclusion dev. case} uses the following lemma.\\

Let $D: V \to \eeu$ be a developing map. 

\begin{lemma}\label{lemma: E(U) included in E(V)}
The inclusion map $i: U \hookrightarrow V$ induces a conformal embedding of $E(U)$ into $E(V)$ that sends every fiber of $E(U)$ on a fiber of $E(V)$ and such that the restriction to every fiber of $E(U)$ is surjective. 
\end{lemma}

\begin{proof}
Consider the foliation of $V$ by inextensible timelike curves induced by the pull-back by the developing map $D$ of the vector field $\partial_t$ on $Ein_{1,n} \simeq \mathbb{S}^{n-1} \times \R$. Since $U$ is causally convex in $V$, the intersection of every leaf of $V$ with $U$ is an inextensible timelike curve of $U$. Therefore, the foliation of $V$ induces a foliation of $U$ by inextensible timelike curves. Notice that this foliation coincides with that induced by the pull-back of $\partial_t$ by the restriction of $D$ to $U$. Let $\psi: V \to \mathcal{B}$ and $\psi_U  : U \to \mathcal{B}_U$ be the canonical projections on the leaf spaces. Then, the map $\psi \circ i$ descends to the quotient in an embedding $\bar{i}: \mathcal{B}_U \to \mathcal{B}$. Let $d: \mathcal{B} \to \mathbb{S}^{n-1}$ (resp. $d_U: \mathcal{B}_U \to \mathcal{S}^{n-1}$) be the developing map induced by the developing map $D$ (resp. the restriction of $D$ to $U$). It is clear that $d \circ \bar{i} = d_U$. It follows that the map from $E(U)$ to $E(V)$ that sends $(p,b)$ on $(p, \bar{i}(b))$ is a conformal embedding that sends every fiber of $E(U)$ on a fiber of $E(V)$. Moreover, the restriction to every fiber of $E(U)$ is clearly surjective.
\end{proof}

In other words, Lemma \ref{lemma: E(U) included in E(V)} says that $E(U)$ can be seen as the union of the fibers of $E(V)$ over some open subset of $\mathcal{B}$.

\begin{proof}[Proof of Proposition \ref{prop: max. ext. respect inclusion dev. case}]
We identify $V$ (resp. $U$) with its image in $E(V)$ (resp. $E(U)$), then we identify $E(U)$ with its image in $E(V)$. Let $\hat{V}$ (resp. $\hat{U}$) be the $\mathcal{C}_0$-maximal extension of $V$ (resp. $U$). By Proposition \ref{prop: max. ext. dev.}, $\hat{V}$ (resp. $\hat{U}$) is the Cauchy development $\mathcal{C}(S)$ of a Cauchy hypersurface $S$ of $V$ (resp. $\Sigma$ of $U$) in $E(V)$ (resp. $E(U)$). \\
We prove that the Cauchy development $\mathcal{C}(\Sigma)$ of $\Sigma$ in $E(V)$ is exactly $\hat{U}$, then we prove that $\mathcal{C}(\Sigma) \subset \mathcal{C}(S)$:

Let $x \in \hat{U}$ and let $\hat{\gamma}$ an inextensible causal curve of $E(V)$ going through $x$. The intersection of $\hat{\gamma}$ with $E(U)$ is a union of connected components. The component containing $x$ is an inextensible causal curve of $E(U)$, then it intersects $\Sigma$ in a single point. Hence, $x \in \mathcal{C}(\Sigma)$. This proves that $\hat{U} \subset \mathcal{C}(\Sigma)$. Actually, this inclusion is a Cauchy embedding. Since $U$ is $\mathcal{C}_0$-maximal, we deduce that $\hat{U} = \mathcal{C}(\Sigma)$.

Let $x \in \mathcal{C}(\Sigma)$ and let $\hat{\gamma}$ an inextensible causal curve of $E(V)$ going through $x$. By definition, the curve $\hat{\gamma}$ meets $\Sigma$, hence $V$. Since $V$ is causally convex in $E(V)$, the intersection of $\hat{\gamma}$ with $V$ is an inextensible causal curve of $V$, thus it intersects $S$ in a single point. It follows that $x \in \mathcal{C}(S)$. Hence, $\mathcal{C}(\Sigma) \subset \mathcal{C}(S)$. The proposition follows.
\end{proof}

\subsection{The general case} \label{sec: general case}

In the previous section, we proved Theorem \ref{main theorem} in the case where $V$ is simply-connected. In this section, we prove it for any conformally flat globally hyperbolic spacetime $V$. Without loss of generality, we suppose that $V$ is $\mathcal{C}_0$-maximal. 
Let $\pi: \tilde{V} \to V$ be the universal cover of $V$. Set $\Gamma := \pi_1(V)$. Let $U'$ be a connected component of $\pi^{-1}(U)$ and let $\Gamma'$ be the stabilizer of $U'$ in $\Gamma$.

\begin{lemma}
The open set $U'$ is causally convex in $\tilde{V}$.
\end{lemma}

\begin{proof}
Let $p,q \in U'$ and let $\gamma$ be a causal curve in $\tilde{V}$ joining $p$ to $q$. The projection $\pi(\gamma)$ of $\gamma$ in $V$ is a causal curve joining the points $\pi(p)$ and $\pi(q)$ of $U$. Since $U$ is causally convex in $V$, the curve $\pi(\gamma)$ is contained in $U$. Hence, $\gamma$ is contained in $U'$. The lemma follows.
\end{proof}

Let $\widehat{U'}$ be the $\mathcal{C}_0$-maximal extension of $U'$. By Proposition \ref{prop: max. ext. respect inclusion dev. case},     $\widehat{U'}$ can be conformally identified to a causally convex open subset of $\tilde{V}$. 

\begin{lemma}
The stabilizer of $\widehat{U'}$ in $\Gamma$ is equal to $\Gamma'$.
\end{lemma}

\begin{proof}
Since $U' \subset \widehat{U'}$, the stabilizer of $U'$ is contained in the stabilizer of $\widehat{U'}$. Conversely, let $\gamma$ be an element of $\Gamma$ stabilizing $\widehat{U'}$. Let $p \in U'$. Consider a Cauchy hypersurface $\Sigma$ of $U'$ going through $p$. Let $\varphi$ be an inextensible causal curve of $\widehat{U'}$ going through $p$. Then, $\gamma \varphi$ is an inextensible causal curve of $\widehat{U'}$. Therefore, $\gamma \varphi$ intersects $\Sigma$ in a unique point $q$. We prove that $q = \gamma p$.

Since $\widehat{U'}$ is causally convex in $\tilde{V}$, the projection $\pi(\widehat{U'})$ is causally convex in $V$ and $\pi(\Sigma)$ is a Cauchy hypersurface of $\pi(\widehat{U'})$. Moreover, $\pi(\varphi)$ is an inextensible causal curve of $\pi(\widehat{U'})$, then it meets $\pi(\Sigma)$ in a single point. Since $\pi(p), \pi(q) \in \pi(\varphi) \cap \pi(\Sigma)$, we deduce that $\pi(p) = \pi(q)$. Then, $q = \gamma' p$ with $\gamma' \in \Gamma$. Since $\gamma p, \gamma' p \in \gamma \varphi$, if $\gamma \not = \gamma'$, the causal curve $\pi(\varphi)$ would be closed. Contradiction. Hence, $\gamma = \gamma'$, so $q = \gamma p$. This shows that $\gamma p \in U'$. Thus, $\gamma \in \Gamma'$.
\end{proof}

The inclusion $\widehat{U'} \subset \tilde{V}$ descends to the quotient in a conformal embedding from $\Gamma' \backslash \widehat{U'}$ to $V$. Since $\widehat{U'}$ is causally convex in $\tilde{V}$, the image of this embedding is causally convex in $V$. It remains to prove the following assertion to conclude.

\begin{lemma}
The quotient space $\Gamma' \backslash \widehat{U'}$ is the $\mathcal{C}_0$-maximal extension of $U$.
\end{lemma}

\begin{proof}
The inclusion $U' \subset \widehat{U'}$ is a Cauchy embedding which descends to the quotient in a Cauchy embedding from $U$ to $\Gamma' \backslash \widehat{U'}$. We have to prove that $\Gamma' \backslash \widehat{U'}$ is $\mathcal{C}_0$-maximal.

Let $f$ be a Cauchy embedding from $\Gamma' \backslash \widehat{U'}$ in a globally hyperbolic conformally flat spacetime $W$. Then, the morphism $f_*: \pi_1(\Gamma' \backslash \widehat{U'}) \to \pi_1(W)$ induced by $f$ is an isomorphism. Let $\pi': \widehat{U'} \to \Gamma' \backslash \widehat{U'}$ be the canonical projection. It induces an injective morphism $\pi'_* : \pi_1(\widehat{U'}) \to \pi_1(\Gamma' \backslash \widehat{U'})$. Hence, the morphism $(f \circ \pi')_* : \pi_1 (\widehat{U'}) \to \pi_1(W)$ is injectif. Let $p: W' \to W$ the cover such that $p_*(\pi_1(W')) = (f \circ \pi')_*(\pi_1(\widehat{U'}))$. Then, $f$ lifts to a Cauchy embedding $f': \widehat{U'} \to W'$. Since $\widehat{U'}$ is $\mathcal{C}_0$-maximal, $f'$ is surjective. Hence, $f$ is surjective. The lemma follows.
\end{proof}

\section{Eikonal functions and $\mathcal{C}_0$-maximality} \label{sec: Eikonal functions}

\subsection{Eikonal functions on the sphere}

In this section, we characterize causally convex open subsets of $\eeu$ which are $\mathcal{C}_0$-maximal in a spatio-temporal decomposition $\mathbb{S}^{n-1} \times \R$. We denote by $d$ the distance on $\mathbb{S}^{n-1}$ induced by the round metric. By \cite[Theorem 10]{Salvemini2013Maximal}, causally convex open subsets of $\eeu$ with conjugate points are conformally equivalent to $\eeu$. For this reason, we consider a causally convex open subset $\Omega$ \emph{without conjugate points}. 

Let $f^+$ and $f^-$ be the two $1$-Lipschitz real-valued functions defined on the projection $U$ of $\Omega$ in $\mathbb{S}^{n-1}$ such that:
\begin{align*}
\Omega &= \{(x,t) \in U \times \R,\ f^-(x) < t < f^+(x)\}.
\end{align*}
We call $f$ the common extension of $f^+$ and $f^-$ to $\partial U$. Let $g^+, g^-$ be the real-valued functions defined for every $x \in U$ by:
\begin{align} \label{g^+}
g^+(x) &= \inf_{x_0 \in \partial U} \{f(x_0) + d(x,x_0)\}   
\end{align}
\begin{align}\label{g^-}
g^-(x) &= \sup_{x_0 \in \partial U} \{f(x_0) - d(x,x_0)\}.
\end{align}

\begin{proposition} \label{prop: characterization of maximal causally convex subsets in a spatio-temporal decomposition}
The causally convex open subset $\Omega$ of $\eeu$ is $\mathcal{C}_0$-maximal if and only if $f^+$ equals $g^+$ and $f^-$ equals $g^-$.
\end{proposition}

The proof of Proposition \ref{prop: characterization of maximal causally convex subsets in a spatio-temporal decomposition} uses the following lemma.

\begin{lemma} \label{tech lemma: maximal causally convex subsets of eeu}
Any $1$-Lipschitz real-valued function $g$ defined on $U$, whose extension to $\partial U$ equals $f$, is bounded by $g^-$ and $g^+$.
\end{lemma}

\begin{proof}
Let $x \in U$ and let $x_0 \in \partial U$. We denote by $\bar{g}$ the extension of $g$ to $\partial U$. Since $\bar{g}$ is $1$-Lipschitz, we have $\bar{g}(x) - \bar{g}(x_0) \leq d(x,x_0)$, i.e. $g(x) - f(x_0) \leq d(x,x_0)$. Hence, $g(x) \leq f(x_0) + d(x,x_0)$. Therefore, $g(x) \leq f^+(x)$. From $-d(x,x_0) \leq \bar{g}(x) - \bar{g}(x_0)$, we deduce similarly that $f^-(x) \leq g(x)$. The lemma follows.
\end{proof}

\begin{proof}[Proof of Proposition \ref{prop: characterization of maximal causally convex subsets in a spatio-temporal decomposition}]

Let us denote $\Omega'$ the set of points $(x,t)$ of $\eeu$ such that $g^-(x) < t < g^+(x)$. By Lemma \ref{tech lemma: maximal causally convex subsets of eeu} and Proposition \ref{prop: Cauchy hypersurfaces of causally convex subsets of Ein}, $\Omega'$ is a Cauchy-extension of $\Omega$. According to Theorem \ref{main theorem}, the $\mathcal{C}_0$-maximal extension of $\Omega$ is conformally equivalent to a causally convex open subset $\hat{\Omega}$ of $\eeu$ such that $\Omega \subset \Omega' \subset \hat{\Omega}$ where each inclusion is a Cauchy embedding. Thus, there exist two $1$-Lipschitz real-valued functions $h^-, h^+$ defined on $U$, whose extensions to $\partial U$ equal $f$, such that $\hat{\Omega}$ is the set of points $(x,t)$ such that $h^-(x) < t < h^+(x)$. Hence, $g^-(x) \leq h^-(x) < h^+(x) \leq g^+(x)$ for every $x \in U$ (see \ref{tech lemma: maximal causally convex subsets of eeu}). In other words, $\hat{\Omega} \subset \Omega'$. The proposition follows.
\end{proof}

The domain $\hat{\Omega}$ in the proof of Proposition \ref{prop: characterization of maximal causally convex subsets in a spatio-temporal decomposition} is a union of connected component of the dual of the graph of $f$ (see Proposition \ref{prop: characterization of the dual}). Therefore, Proposition \ref{prop: characterization of maximal causally convex subsets in a spatio-temporal decomposition} can be reformulated as: \emph{$\mathcal{C}_0$-maximal causally convex open subsets of $\eeu$ are exactly unions of connected components of duals of closed achronal subsets of $\eeu$.}

\begin{definition}
The function $g^+$ (resp. $g^-$) is called \emph{future eikonal} (resp. \emph{past eikonal}).
\end{definition}

\begin{remark}
This definition is motivated by the classical notion of eikonal function in analysis: a real-valued function $f$ defined on an open subset $U$ of $\R^n$ is called eikonal if $f$ is differentiable almost everywhere and satisfies the eikonal equation $||\triangledown f|| = 1$. It turns out that if $U$ has a piecewise smooth boundary, then the function $f(x) := d(x, \partial U)$, where $d$ is the usual distance on $\R^n$, is eikonal in this sense.
\end{remark}

The following proposition gives a geometrical characterisation of eikonal functions.

\begin{proposition}[Geometrical criterion of eikonality]\label{prop: geometric charac. of eikonal functions}
A real-valued function $f^+$ defined on an open subset $U$ of the sphere $\mathbb{S}^{n-1}$ is future eikonal if and only if for every $x \in U$, there exists a past-directed lightlike geodesic starting from $(x, f^+(x))$, entirely contained in the graph of $f^+$ and with no past endpoint in the graph of $f^+$.
\end{proposition}

\begin{proof}
Suppose $f^+$ is eikonal. We call $f$ the extension of $f^+$ to $\partial U$. Let $x \in U$. There exists $x_0 \in \partial U$ such that $f^+(x) = f(x_0) + d(x,x_0)$. Hence, $f^+(x) -  f(x_0) = d(x,x_0)$, i.e. the points $(x, f^+(x))$ and $(x_0, f(x_0))$ are joined by a past lightlike geodesic $\varphi$. This last one is contained in the graph of $f^+$. Conversely, let us prove that $f^+$ equals the function $g^+$ defined by (\ref{g^+}). By Lemma \ref{tech lemma: maximal causally convex subsets of eeu}, we have $f^+ \leq g^+$. Let $x \in U$. There exists a past-directed lightlike geodesic $\varphi$ starting from $(x, f^+(x))$, entirely contained in the graph of $f^+$ and with no past endpoint in the graph of $f^+$. In $\eeu$, the geodesic $\varphi$ admits a past endpoint $(x_0, f(x_0))$ with $x_0 \in \partial U$. Thus, $d(x,x_0) = f^+(x) -  f(x_0)$. Hence, $f^+(x) = f(x_0) + d(x,x_0) \geq g^+(x)$. Then, $f^+ = g^+$.
\end{proof}

Now, we show that eikonal is \emph{a local property}.

\begin{definition}
A real-valued function $f^+$ defined on an open subset $U$ of the sphere $\mathbb{S}^{n-1}$ is \emph{locally future eikonal} if every point $x$ of $U$ admits an arbitrarily small neighborhood $V_x$ such that the restriction of $f^+$ to $V_x$ is future eikonal.
\end{definition}

\begin{remark}
A locally future eikonal function is locally $1$-Lipschitz, thus $1$-Lipschitz.
\end{remark}

Let $f^+$ be a $1$-Lipschitz real-valued function defined on an open subset $U$ of $\mathbb{S}^{n-1}$. We call $f$ its extension to $\partial U$.

\begin{lemma}
If $f^+$ is future eikonal then $f^+$ is locally future eikonal.
\end{lemma}

\begin{proof}
Let $V$ be an open subset of $U$. We call $g$ the extension of $f^+_{|V}$ to $\partial V$. Let $g^+$ (resp. $g^-$) be the function defined by the expression (\ref{g^+}) (resp. (\ref{g^-}) after replacing $U$ by $V$. By Lemma \ref{tech lemma: maximal causally convex subsets of eeu}, we have $f^+_{|V} \leq g$. Let us prove that $f^+_{|V} \geq g$ so we get the equality. 
Let $\Omega$ (resp. $W$) be the $\mathcal{C}_0$-maximal causally convex open subset of $\eeu$ bounded by the graphs of $f^+$ and $f^-$ (resp. $g^+$ and $g^-$). The intersection $W \cap \Omega$ is a causally convex open subset of $\Omega$; its $\mathcal{C}_0$-maximal extension is $W$. By theorem \ref{main theorem}, we get $W \subset \Omega$. Hence, $g^+ \leq f^+_{|V}$. Thus, $g^+ = f^+_{|V}$, then $f^+_{|V}$ is future eikonal. 
\end{proof}

\begin{lemma}
If $f^+$ is locally future eikonal then $f^+$ is future eikonal.
\end{lemma}

\begin{proof}
We use the criterion given by Proposition \ref{prop: geometric charac. of eikonal functions} to prove that $f^+$ is eikonal. Let $x \in U$ and let $V \subset U$ be a neighborhood of $x$ such that $f^+_{|V}$ is future eikonal. Then, there is a past-directed lightlike geodesic $\varphi$ starting from $(x, f^+(x))$ with the properties of Proposition \ref{prop: geometric charac. of eikonal functions}. In $\eeu$, the geodesic $\varphi$ admits a past endpoint $(x_0, f^+(x_0))$ where $x_0 \in \partial V$ (we still denote by $f^+$ its extension to $\overline{U}$). If $x_0 \in \partial U$, the lemma is proved. Otherwise, we choose a neighborhood $V_0 \subset U$ of $x_0$ such that $f^+_{|V_0}$ is future eikonal. Again, there is a past-directed lightlike geodesic $\varphi_0$ starting from $(x_0, f(x_0))$, entirely contained in the graph of $f^+_{|V_0}$, with past endpoint $(x_1, f^+(x_1))$ where $x_1 \in \partial V_0$. The geodesic $\varphi_0$ extends $\varphi$; indeed, otherwise $(x_1, f^+(x_1))$ would be in the chronological past of $(x, f^+(x))$, i.e. we would have $d(x,x_1) < f^+(x) - f^+(x_1)$. This contradicts the fact that $f^+$ is $1$-Lipschitz. Consequently, we extend $\varphi$ in a lightlike geodesic, entirely contained in the graph of $f^+$, with no past endpoint in the graph of~$f^+$. The lemma follows.
\end{proof}

We proved the following statement:

\begin{proposition}
A real-valued function defined on an open subset of the sphere $\mathbb{S}^{n-1}$ is future eikonal if and only if it is \emph{locally} future eikonal. \qed
\end{proposition}

\subsection{Eikonal functions on a conformally flat Riemannian manifold}

Let $\mathcal{B}$ be a conformally flat Riemannian manifold of dimension $(n-1) \geq 2$ and let $d: \mathcal{B} \to \mathbb{S}^{n-1}$ be a developing map. The notion of eikonal function on an open subset of the sphere $\mathbb{S}^{n-1}$ being \emph{local}, it naturally generalizes to functions on an open subset of $\mathcal{B}$:

\begin{definition}
A real-valued function $f^+$ defined on an open subset $U$ of $\mathcal{B}$ is said \emph{future eikonal} if for every $x \in U$, there exists an open neighborhood $U_x$ of $x$ contained in $U$ such that the following hold:
\begin{enumerate}
\item the restriction of $d$ to $U_x$ is injective;
\item the map $f^+ \circ d_{|U_x}^{-1}: d(U_x) \subset \mathbb{S}^{n-1} \to \R$ is future eikonal.
\end{enumerate}
\end{definition}

Past eikonal functions on $U$ are defined similarly with the reverse time orientation. The geometrical criterion of eikonality given in Proposition \ref{prop: geometric charac. of eikonal functions} is still valid in this setting. \\

Eikonality is closely related to $\mathcal{C}_0$-maximality. Indeed, let $E$ be the pull-back of the trivial fiber bundle $\pi: \eeu \to \mathbb{S}^{n-1}$ by $d$:
\begin{align*}
E &:= \{(p,b) \in \eeu \times \mathcal{B}:\ \pi(p) = d(b)\}.
\end{align*} 
We denote by $\pi: E \to \mathcal{B}$ the projection on the second factor. The construction of this fiber bundle is exactly the same than that of the enveloping space in Section \ref{sec: def enveloping space}. Then, $E$ is a conformally flat spacetime of dimension $n \geq 3$ sharing the same properties than that of the enveloping space. Let $\hat{D}: E \to \eeu$ be the developing map defined as the projection on the first factor. A causally convex GH open subset $\Omega$ of $E$ is given by
\begin{align*}
\Omega &= \{(x,t) \in U \times \R:\ f^-(x) < t < f^+(x)\}
\end{align*}
where $f^+$ and $f^-$ are $1$-Lipschitz functions defined on an open subset $U$ of $\mathcal{B}$ such that their extensions to $\partial U$ coincide (see Section \ref{sec: Causally convex subsets of E(M)}). 

\begin{proposition}
The GH conformally flat spacetime $\Omega$ is $\mathcal{C}_0$-maximal if and only if $f^+$ is future-eikonal and $f^-$ past-eikonal.
\end{proposition}

\begin{proof}
Suppose $\Omega$ is $\mathcal{C}_0$-maximal. Let $x \in U$. Set $p_+ = (x, f^+(x))$. Let $p \in I^-(p^+) \cap \Omega$. Since $\Omega$ is GH, the restriction of $\hat{D}$ to $I^+(p, \Omega)$ is injective and its image is causally convex in $\eeu$ (see \cite[Prop. 2.7 and Cor. 2.8, p.151]{salveminithesis}). The $\mathcal{C}_0$-maximality of $\Omega$ implies that $\hat{D}(I^+(p,\Omega))$ is also $\mathcal{C}_0$-maximal (see \cite[Prop. 3.6, p.156]{salveminithesis}). It follows that:
\begin{itemize}
\item $\hat{D}(I^+(p, \Omega))$ is the domain bounded by the graph of a future-eikonal function $g^+$ and the graph of a past-eikonal function $g^-$ defined on an open subset of $\mathbb{S}^{n-1}$ (see Proposition \ref{prop: characterization of maximal causally convex subsets in a spatio-temporal decomposition});
\item $d$ is injective on $U_x := \hat{\pi}(I^+(p,\Omega))$ and $g^+ = f^+ \circ d_{|U_x}^{-1}$.
\end{itemize}
Thus, $f^+$ is future-eikonal. The proof is similar for $f^-$ with the reverse time-orientation.

Conversely, suppose that $f^+$ is future-eikonal and $f^-$ past-eikonal. Let $S$ be a Cauchy hypersurface in $\Omega$. By Proposition \ref{prop: max. ext. dev.}, the $\mathcal{C}_0$-maximal extension of $\Omega$ is conformally equivalent to a causally convex open subset $\hat{\Omega}$ of $E$, containing $\Omega$ and for which $S$ is a Cauchy hypersurface. Suppose that $\Omega$ is stricly contained in $\hat{\Omega}$. Then, there exists $x \in U$ such that $(x, f^+(x)) \in \hat{\Omega}$ or $(x, f^-(x)) \in \hat{\Omega}$. Suppose that $(x, f^+(x)) \in \hat{\Omega}$. The proof is symmetric if $(x, f^-(x)) \in \hat{\Omega}$. Since $f^+$ is future-eikonal, there exists a past-directed lightlike geodesic $\varphi$ starting from $(x, f^+(x))$, entirely contained in the graph of $f^+$ and with no endpoint in the graph of $f^+$. Then, $\varphi$ does not intersects $S$. Contradiction. Hence, $\Omega = \hat{\Omega}$, in other words $\Omega$ is $\mathcal{C}_0$-maximal.
\end{proof}

\bibliographystyle{plain}
\bibliography{bibliography}

\begin{thebibliography}{10}

\bibitem{andersson2012}
Lars Andersson, Thierry Barbot, François Béguin, and Abdelghani Zeghib.
\newblock Cosmological time versus {CMC} time in spacetimes of constant
  curvature.
\newblock {\em Asian Journal of Mathematics}, 16(1):37--88, 03 2012.

\bibitem{BernalSanchez}
Antonio Bernal and Miguel Sánchez.
\newblock On smooth cauchy hypersurfaces and geroch's splitting theorem.
\newblock {\em Communications in Mathematical Physics}, 243:461--470, 12 2003.

\bibitem{bernal2007globally}
Antonio~N Bernal and Miguel S{\'a}nchez.
\newblock Globally hyperbolic spacetimes can be defined as ‘causal’instead
  of ‘strongly causal’.
\newblock {\em Classical and Quantum Gravity}, 24(3):745, 2007.

\bibitem{francesarticle}
Charles Frances.
\newblock Une preuve du th\'{e}or\`eme de {L}iouville en g\'{e}om\'{e}trie
  conforme dans le cas analytique.
\newblock {\em Enseign. Math. (2)}, 49(1-2):95--100, 2003.

\bibitem{frances2005lorentzian}
Charles Frances.
\newblock Lorentzian kleinian groups.
\newblock {\em Commentarii Mathematici Helvetici}, 80(4):883--910, 2005.

\bibitem{Geroch1970}
Robert~Paul Geroch.
\newblock The domain of dependence.
\newblock {\em Journal of Mathematical Physics}, 11:437--449, 1970.

\bibitem{goldman}
William Goldman.
\newblock Geometric structures on manifolds.
\newblock {\em University of Maryland, Lecture Notes}, 2018.

\bibitem{oneill}
Barrett O'neill.
\newblock {\em Semi-Riemannian geometry with applications to relativity}.
\newblock Academic press, 1983.

\bibitem{salveminithesis}
Clara~Rossi Salvemini.
\newblock {\em Espace-temps globalement hyperboliques conform{\'e}ment plats}.
\newblock PhD thesis, Universit{\'e} d'Avignon, 2012.

\bibitem{Salvemini2013Maximal}
Clara~Rossi Salvemini.
\newblock Maximal extension of conformally flat globally hyperbolic spacetimes.
\newblock {\em Geometriae Dedicata}, 174:235--260, 2013.

\bibitem{smai2022anosov}
Rym Smai.
\newblock Anosov representations as holonomies of globally hyperbolic spatially
  compact conformally flat spacetimes.
\newblock {\em Geometriae Dedicata}, 216(4):1--36, 2022.

\bibitem{Choquet-Bruhat}
R.~Geroch Y.~Choquet-Bruhat.
\newblock Global aspects of the cauchy problem in general relativity.
\newblock {\em Commun.Math.Phys.14,329-335}, (1969).

\end{thebibliography}

\end{document}